\documentclass{article}
\usepackage[T1]{fontenc}
\usepackage{amsmath,amssymb,amsthm,mathrsfs,calc,graphicx,pdfpages,bm,tikz}
\usepackage[left=2.5cm, top=2.5cm,bottom=2.5cm,right=2.5cm]{geometry}

 \usepackage{mathpazo}
 
 \newtheorem{theorem}{Theorem}[section]
 
 \newtheorem{lemma}[theorem]{Lemma}
 
 \newtheorem{remark}[theorem]{Remark}
 
 \def\ba{\begin{array}}
 \def\ea{\end{array}}
 \def\bea{\begin{eqnarray} \label}
 \def\eea{\end{eqnarray}}
 \def\be{\begin{equation} \label}
 \def\ee{\end{equation}}
 \def\bit{\begin{itemize}}
 \def\eit{\end{itemize}}
 \def\ben{\begin{enumerate}}
 \def\een{\end{enumerate}}

 \def\EE{\mathbb{E}}

 \def\NN{\mathbb{N}}

 \def\RR{\mathbb{R}}
 \def\RRd{\mathbb{R}^{d}}

 \def\d{\delta}

 \def\k{\kappa}

 \def\G{\Gamma}
  \def\bfG{{\bf \Gamma}}

 \def\Sig{\Sigma}

 \def\bD{\mathbf{D}}
 
  \def\bF{\mathbf{F}}

 \def\bN{\mathbf{N}}

 \def\bV{\mathbf{V}}
  \def\bW{\mathbf{W}}
 \def\bX{\mathbf{X}}
  \def\bY{\mathbf{Y}}
 
 \def\cC{\mathcal{C}}

 \def\cG{\mathcal{G}}

 \def\cL{\mathcal{L}}

 \def\cP{\mathcal{P}}

 \def\dint{\textup{d}}

 \def\vol{\textup{vol}}

 \def\aut{\textup{aut}}
\def\cov{\textup{cov}}

 \DeclareMathOperator{\Var}{Var}

 \DeclareMathOperator{\R}{\mathbb{R}}
 
 \DeclareMathOperator{\1}{{\bf 1}}
 
 \DeclareMathOperator{\dom}{dom}

 \newcommand{\ellnorm}[3]{\Vert{#3}\Vert_{\ell^{#1}(\NN)^{\otimes{#2}}}}

 \newcommand{\absolute}[1]{\vert{#1}\vert}
 \newcommand{\Bigabsolute}[1]{\Big\vert{#1}\Big\vert}
 \newcommand{\norm}[1]{\|{#1}\|}
 \newcommand{\Bignorm}[1]{\Big\|{#1}\Big\|}

\begin{document}

\title{\bfseries Multivariate central limit theorems for Rademacher functionals with applications}

\author{Kai Krokowski\footnotemark[1]\; and Christoph Th\"ale\footnotemark[2]}

\date{} \renewcommand{\thefootnote}{\fnsymbol{footnote}}

\footnotetext[1]{Ruhr University Bochum, Faculty of Mathematics, NA 3/28, D-44780 Bochum, Germany. E-mail: kai.krokowski@rub.de}

\footnotetext[2]{Ruhr University Bochum, Faculty of Mathematics, NA 3/68, D-44780 Bochum, Germany. E-mail: christoph.thaele@rub.de}

\maketitle

\begin{abstract}
Quantitative multivariate central limit theorems for general functionals of possibly non-symmetric and non-homogeneous infinite Rademacher sequences are proved by combining discrete Malliavin calculus with the smart path method for normal approximation. In particular, a discrete multivariate second-order Poincar\'e inequality is developed. As a first application, the normal approximation of vectors of subgraph counting statistics in the Erd\H{o}s-R\'enyi random graph is considered. In this context, we further specialize to the normal approximation of vectors of vertex degrees. In a second application we prove a quantitative multivariate central limit theorem for vectors of intrinsic volumes induced by random cubical complexes.
\end{abstract}

\section{Introduction}

Suppose that $X=(X_k)_{k\in\NN}$ is a Rademacher sequence, that is, a sequence of independent random variables satisfying, for all $k\in\NN$, $P(X_k=1)=p_k$ and $P(X_k=-1)=q_k=1-p_k$ for some $p_k\in(0,1)$. Further, fix a dimension parameter $d\in\NN$ and let $F_1=F_1(X),\ldots,F_d=F_d(X)$ be $d$ random variables depending on possibly infinite many members of the Rademacher sequence $X$. We shall refer to such random variables as Rademacher functionals in what follows. The goal of this paper is to derive handy conditions under which the random vector $\bF=(F_1,\ldots,F_d)$ consisting of $d$ Rademacher functionals is close in distribution to a $d$-dimensional Gaussian random vector. In our paper the distributional closeness will be measured by means of a multivariate probability metric based on four times partially differentiable test functions. We will provide two versions of such a result. One is in the spirit of the Malliavin-Stein method and expresses the distributional closeness in terms of so-called discrete Malliavin operators. The second one is a multivariate discrete second-order Poincar\'e inequality, a bound which only involves the first- and second-order discrete Malliavin derivatives of the Rademacher functionals $F_1,\ldots,F_d$, or, more precisely, their moments up to order four. More formally, if $F=F(X)$ is a Rademacher functional, the discrete Malliavin derivative $D_kF$ in direction $k\in\NN$ is defined as $D_kF=\sqrt{p_kq_k}(F_k^+-F_k^-)$, where $F_k^\pm$ is the Rademacher functional for which the $k$th coordinate $X_k$ of the Rademacher sequence $X$ is conditioned to be $\pm 1$. The second-order discrete derivative is iteratively given by $D_kD_\ell F$ for $k,\ell\in\NN$. Such a bound is particularly attractive for concrete applications as demonstrated in the present text.

\smallskip

Let us describe the purpose and the content of our paper in some more detail.
\begin{itemize}
\item[(i)] First of all, our aim is to provide a multivariate quantitative central limit theorem for vectors of Rademacher functionals by bringing together the discrete Malliavin calculus of variations with the so-called smart-path method for normal approximation. This leads to a limit theorem in the spirit of the Mallavin-Stein method and generalizes an earlier result from \cite{KRT1}, where the underlying Rademacher sequence has been assumed to be homogeneous and symmetric, meaning that $p_k=q_k=1/2$ for all $k\in\NN$ in above notation. 

\item[(ii)] From this result, a further aim of this text is to develop a discrete multivariate second-order Poincar\'e inequality, that is, a bound for the multivariate normal approximation that only involves the first- and second-order discrete Malliavin derivatives, or, more precisely, its moments up to order four. Such a result can be regarded as the multivariate analogue of the main theorem obtained in \cite{KRT2}.

\item[(iii)] Finally, we want to demonstrate the flexibility and applicability of our discrete multivariate second-order Poincar\'e inequality by means of examples from the theory of random graphs and random topology. First, we are going to provide a bound of order $O(n^{-1})$ for the multivariate normal approximation of a vector of subgraph counts in the classical Erd\H{o}s-R\'enyi random graph. This generalizes (in a different probability metric) a result of Reinert and R\"ollin \cite{ReinertRollin}, where vectors of the number of edges, $2$-stars and triangles have been considered, and adds a rate of convergence to the related central limit theorem in the paper of Janson and Nowicki \cite{JansonNowicki}. Moreover, for the same model we also provide a multivariate central limit theorem for the random vector of vertex degrees with a rate of convergence of order $O(n^{-1/2})$. This can be seen as a version of the result of Goldstein and Rinott \cite{GoldsteinRinott} and is the multivariate analogue of a related Berry-Esseen bound proved by Goldstein \cite{Goldstein} and Krokowski, Reichenbachs and Th\"ale \cite{KRT2}. Second, we consider the vector of intrinsic volumes determined by different models of random cubical complexes in $\RRd$ and derive bounds of order $O(n^{-d/2})$ on the error in their normal approximation. This constitutes a multivariate extension of the central limit theorem provided by Werman and Wright \cite{WermanWright} and is in line with recent developments in the active field of random topology, see \cite{AdlerEtAl,BobrowskiKahle,KahleSurvey,LinialMeshulam} as well as the references cited therein.
\end{itemize}

\smallskip

Our results continue a recent line of research concerning limit theorems for Rademacher functionals. The field has been opened by Nourdin, Peccati and Reinert \cite{NourdinPeccatiReinertRademacher}, who proved first limit theorems for a class of smooth probability metrics. Later, Krokowski, Reichenbachs and Th\"ale \cite{KRT1,KRT2} considered Berry-Esseen bounds and provided a first univariate discrete second-order Poincar\'e inequality. Zheng \cite{Zheng} has obtained a refined bound for the Wasserstein distance and also proved almost sure central limit theorems. Moreover, Privault and Torrisi \cite{PrivaultTorrisi} as well as Krokowski \cite{Krokowski}  also derived bounds for the Poisson approximation of Rademacher functionals.

\smallskip

This text is organized as follows. In Section \ref{sec:MalliavinBackground} we briefly recall the basis of discrete Malliavin calculus in order to keep the paper reasonably self-contained. A first quantitative multivariate central limit theorem for functionals of a possibly non-symmetric and non-homogeneous infinte Rademacher sequence based on the discrete Malliavin-Stein method is presented in Section \ref{sec:1stBound}, while Section \ref{sec:2ndOrderPoincare} contains the discrete multivariate second-order Poincar\'e inequality. The applications to subgraph and vertex degree counts in the Erd\H{o}s-R\'enyi random graph and to the intrinsic volumes of random cubical complexes are discussed in the final Section \ref{sec:Application}.

\section{Discrete Malliavin calculus}\label{sec:MalliavinBackground}

In this section we briefly recall the basis of discrete Malliavin calculus. We refer to the monograph \cite{PrivaultLMN} as well as to the papers \cite{KRT1,KRT2,NourdinPeccatiReinertRademacher} for details, proofs and further references.

\paragraph{Rademacher sequences.}
Let $p:=(p_k)_{k \in \NN}$ be a sequence of success probabilities $0 < p_k <1$ and put $q:=(q_k)_{k \in \NN}$ with $q_k := 1-p_k$. Furthermore, let $(\Omega, \mathcal{F}, P)$ be the following probability space: $\Omega:=\{-1,+1\}^{\NN}$, $\mathcal{F}:={\rm power}(\{-1,+1\})^{\otimes \NN}$, where ${\rm power}(\,\cdot\,)$ denotes the power set of the argument set, and $P:=\bigotimes_{k=1}^\infty (p_k \delta_{+1} + q_k \delta_{-1})$ with $\delta_{\pm 1}$ being the unit-mass Dirac measure at $\pm 1$. We let $X:=(X_k)_{k\in\NN}$ be a sequence of independent random variables defined on $(\Omega, \mathcal{F}, P)$ by $X_k(\omega):=\omega_k$, for every $k \in \NN$ and $\omega := (\omega_k)_{k\in\NN} \in \Omega$. We refer to such a sequence $X$ as (possibly non-symmetric and non-homogeneous infinite) Rademacher sequence. We also define the standardized sequence $Y := (Y_k)_{k\in\NN}$ by putting $Y_k:=(\Var(X_k))^{-1/2}(X_k-\EE[X_k])=(2\sqrt{p_kq_k})^{-1}(X_k-p_k+q_k)$ for every $k \in \NN$. Note that $Y_k=X_k$, iff $X$ is a homogeneous and symmetric Rademacher sequence, that is, if $p_k=q_k=1/2$, for all $k \in \NN$.

\paragraph{Discrete multiple stochastic integrals and chaos decomposition.}
Let us denote by $\kappa$ the counting measure on $\NN$. We put $\ell^2(\NN)^{\otimes n} := L^2(\NN^n,\mathcal{P}(\NN)^{\otimes n}, \kappa^{\otimes n})$ for every $n \in \NN$ and refer to the elements of that space as kernels. Let $\ell^2(\NN)^{\circ n}$ denote the subset of $\ell^2(\NN)^{\otimes n}$ consisting of symmetric kernels and let $\ell_0^2(\NN)^{\otimes n}$ be the subset of kernels vanishing on diagonals, that is, vanishing on the complement of the set $\Delta_n:=\{ (i_1, \dotsc, i_n) \in \NN^n : \text{$i_j \neq i_k$ for $j \neq k$} \}$. We then put $\ell_0^2(\NN)^{\circ n} := \ell^2(\NN)^{\circ n} \cap \ell_0^2(\NN)^{\otimes n}$.

For $n \in \NN$ and $f \in \ell_0^2(\NN)^{\circ n}$, we define the discrete multiple stochastic integral of order $n$ of $f$ by
\begin{align*}
J_n(f) &:= \sum_{(i_1, \dotsc, i_n)\in\NN^n} f(i_1, \dotsc, i_n)Y_{i_1} \cdot \dotsc \cdot Y_{i_n} = \sum_{(i_1, \dotsc, i_n)\in\Delta_n} f(i_1, \dotsc, i_n)Y_{i_1} \cdot \dotsc \cdot Y_{i_n} \notag\\
&\phantom{:}= n!\sum_{1 \leq i_1 < \dotsc < i_n < \infty} f(i_1, \dotsc, i_n)Y_{i_1} \cdot \dotsc \cdot Y_{i_n}\,.
\end{align*}
In addition, we put $\ell^2(\NN)^{\otimes 0} := \R$ and $J_0(c):=c$, for every $c \in \RR$.

It is an important fact that every $F \in L^2(\Omega)$ admits a decomposition of the form
\begin{align}\label{Chaos representation}
F = \EE[F] + \sum_{n=1}^\infty J_n(f_n)
\end{align}
with uniquely determined kernels $f_n \in \ell_0^2(\NN)^{\circ n}$.

\paragraph{Discrete Malliavin derivative.} For every $\omega = (\omega_1, \omega_2, \dotsc) \in \Omega$ and $k \in \NN$ we define the two sequences $\omega_\pm^k$ by putting $\omega_\pm^k:=(\omega_1, \dotsc, \omega_{k-1}, \pm1, \omega_{k+1}, \dotsc)$.
Furthermore, for every $F \in L^1(\Omega)$, $\omega \in \Omega$ and $k \in \NN$ let $F_k^\pm(\omega) := F(\omega_\pm^k)$. For such an $F$ the discrete Malliavin derivative is defined by $DF:=(D_kF)_{k \in \NN}$ with
\begin{align}\label{Pathwise gradient}
D_kF:=\sqrt{p_kq_k}\,(F_k^+-F_k^-)\,,\qquad k\in\NN\,.
\end{align}
Note that it immediately follows from \eqref{Pathwise gradient} that, for every $k \in \NN$, $D_kF$ is independent of $X_k$. In the following we state a product formula for the discrete Mailliavin derivative. If $F,G \in L^1(\Omega)$, then
\begin{align}\label{eq:ProductFormula}
D_k(FG) = (D_kF)G+F(D_kG)-\frac{X_k}{\sqrt{p_kq_k}}(D_kF)(D_kG)\,,\qquad k\in\NN\,.
\end{align}
For $m\in\NN$ let us further define the iterated discrete Malliavin derivative of order $m$ of $F$ by $D^mF := (D_{k_1, \dotsc, k_m}^mF)_{k_1, \dotsc, k_m \in \NN}$ with $D_{k_1, \dotsc, k_m}^mF := D_{k_m}(D_{k_1, \dotsc, k_{m-1}}^{m-1}F)$, for every $k_1, \dotsc, k_m \in \NN$, where $D^0F:=F$. Given $F \in L^2(\Omega)$ with chaos representation $F=\EE[F]+\sum_{n=1}^\infty J_n(f_n)$ as in \eqref{Chaos representation} and $m \in \NN$, we will say that $F \in \dom(D^m)$, provided that
\begin{align*}
\EE [ \ellnorm{2}{m}{D^mF}^2 ] = \sum_{n=m}^\infty \frac{n!}{(n-m)!} n!  \ellnorm{2}{n}{f_n}^2 < \infty\,.
\end{align*}
If $F\in\dom(D)$ with chaos decomposition \eqref{Chaos representation}, $D_kF$ can be $P$-almost surely be identified with the random variable given by
$$
D_kF=\sum_{n=1}^\infty nJ_{n-1}(f_n(\,\cdot\,,k))\,,
$$
where $f_n(\,\cdot\,,k)$ stands for the kernel $f_n$ with one of its variables fixed to be $k$ (which one is irrelevant, since the kernels are symmetric).

\paragraph{Discrete divergence.} We will now define the discrete divergence operator $\delta$ and its domain $\dom(\delta)$. Let $f_{n} \in \ell_0^2(\NN)^{\circ n-1} \otimes \ell^2(\NN)$, for every $n \in \NN$, and consider the sequence $u:=(u_k)_{k \in \NN}$ given by $u_k:= \sum_{n=1}^\infty J_{n-1}(f_n(\, \cdot \, ,k))$ for every $k \in \NN$. For such $u$, we say that $u \in \dom(\delta)$, if
\begin{align*}
\sum_{n=1}^\infty n! \ellnorm{2}{n}{\widetilde{f}_n \1_{\Delta_n}}^2 < \infty\,,
\end{align*}
where $\widetilde{f}_n$ denotes the canonical symmetrization of $f_n$. For $u \in \dom(\delta)$, the discrete divergence operator $\delta$ is then defined by
\begin{align*}
\delta(u):= \sum_{n=1}^\infty J_{n}(\widetilde{f}_n \1_{\Delta_n})\,.
\end{align*}
One can interpret $\delta$ as the operator that is adjoint to the discrete Malliavin derivative. Namely, if $F \in \dom(D)$ and $u \in \dom(\delta)$, then
\begin{equation}\label{eq:adjoint}
\EE[F \delta(u)] = \EE[\langle DF, u \rangle_{\ell^2(\NN)}]\,.
\end{equation}

\paragraph{Discrete Ornstein-Uhlenbeck operator and its inverse.}
Next, we define the discrete Ornstein-Uhlenbeck operator $L$ and its (pseudo-)inverse $L^{-1}$. Given $F \in L^2(\Omega)$, again with chaos representation $F=\EE[F]+\sum_{n=1}^\infty J_n(f_n)$ as above, we say that $F \in \dom(L)$, if
\begin{align*}
\sum_{n=1}^\infty n^2 n!  \ellnorm{2}{n}{f_n}^2 < \infty\,.
\end{align*}
For $F \in \dom(L)$, the discrete Ornstein-Uhlenbeck operator $L$ is then defined by
\begin{align*}
LF := -\sum_{n=1}^\infty n J_n(f_n)\,.
\end{align*}
For centred $F \in L^2(\Omega)$, its (pseudo-) inverse is given as follows:
\begin{align*}
L^{-1}F := -\sum_{n=1}^\infty \frac{1}{n} J_n(f_n)\,.
\end{align*}

\paragraph{Discrete Ornstein-Uhlenbeck semigroup.}
Finally, we introduce the semigroup associated with the discrete Ornstein-Uhlenbeck operator $L$. The discrete Ornstein-Uhlenbeck semigroup $(P_t)_{t \geq 0}$ is defined by
\begin{align*}
P_tF := \EE[F] + \sum_{n=1}^\infty e^{-nt} J_n(f_n)\,,\qquad t\geq0\,.
\end{align*}
The process associated with the discrete Ornstein-Uhlenbeck semigroup is given as follows. For every $k \in \NN$, let $X_k^*$ be an independent copy of $X_k$. Furthermore, let $(Z_k)_{k \in \NN}$ be a sequence of independent and exponentially distributed random variables with mean $1$, where $Z_k$ is independent of $X_k$ and $X_k^*$, for every $k \in \NN$. For every real $t \geq 0$, let $X^t:=(X_k^t)_{k\in\NN}$ with
\begin{align*}
X_k^t := X_k^* \1_{\{ Z_k \leq t \}} + X_k \1_{\{ Z_k > t \}}\,,\qquad k\in\NN\,.
\end{align*} 
Then, $(X^t)_{t \geq 0}$ is the discrete Ornstein-Uhlenbeck process associated with the Ornstein-Uhlenbeck semigroup $(P_t)_{t \geq 0}$. The relation of Ornstein-Uhlenbeck semigroup and process is exhibited in the following formula, known as Mehler's formula. If $F\in L^2(\Omega)$, then it $P$-almost surely holds that
\begin{align}\label{eq:Mehler}
P_tF = \EE[F(X^t) \, \vert \, X]\,,\qquad t\geq0\,.
\end{align}

\paragraph{Integration by parts, integrated Mehler's formula and Poincar\'e inequality.}
We notice that the discrete Malliavin operators $D$, $\delta$ and $L$ are related by the identity $L = -\delta D$. Moreover, the following discrete integration by parts formula is valid. If $F,G \in \dom(D)$, then
\begin{align}\label{Integration by parts formula}
\EE[(F-\EE[F])G] = \EE[\langle -DL^{-1}(F - \EE[F]), DG \rangle_{\ell^2(\NN)}]\,.
\end{align}
Indeed, the relation $L = -\delta D$ and the adjointness of $D$ and $\delta$ in \eqref{eq:adjoint} yield 
\begin{align*}
&\EE[(F-\EE[F])G] = \EE[LL^{-1}(F - \EE[F])G]= \EE[-\delta D L^{-1}(F - \EE[F])G] = \EE[\langle -DL^{-1}(F - \EE[F]), DG \rangle_{\ell^2(\NN)}]\,.
\end{align*}
The following identity can be seen as an integrated version of Mehler's formula. If $m, k_1, \dotsc, k_m \in \NN$ and $F \in \dom(D^m)$ with $\EE[F]=0$, then it $P$-almost surely holds that
\begin{align}\label{eq:IntegratedMehler}
-D_{k_1, \dotsc, k_m}^m L^{-1}F = \int_0^\infty e^{-mt} P_tD_{k_1, \dotsc, k_m}^mF \, dt\,.
\end{align}
From this, one can immediately deduce the following important inequality. If $m, k_1, \dotsc, k_m \in \NN$, $\alpha \geq 1$ and $F \in \dom(D^m)$ with $\EE[F]=0$, then
\begin{align}\label{eq:MehlerInequality}
\EE[\absolute{D_{k_1, \dotsc, k_m}^m L^{-1}F}^\alpha] \leq \EE[\absolute{D_{k_1, \dotsc, k_m}^m F}^\alpha]\,.
\end{align}
Finally, let us recall a discrete version of the classical Poincar\'e inequality. For every $F \in L^1(\Omega)$, it holds that
\begin{align}\label{eq:PoincareInequality}
\Var(F) \leq \EE[\norm{DF}_{\ell^2(\NN)}^2]\,.
\end{align}

\section{Multivariate central limit theorems}

\subsection{A discrete Malliavin-Stein bound}\label{sec:1stBound}

In the following, we will prove a bound on the error in the multivariate normal approximation of vectors of general functionals of possibly non-symmetric and non-homogeneous infinite Rademacher sequences. This way we generalize Theorem 5.1 in \cite{KRT1}, where only functionals of symmetric Rademacher sequences have been considered. The proof proceeds along the lines of \cite{KRT1}, but there are a number of subtleties arising in the more general case here that were not present before. In particular, in the non-symmetric case a new summand in the error bound becomes visible as further discussed in Remark \ref{rem:NewTerm} below. To make this and other phenomena transparent, we include the full details.

The distance between the law of a vector of Rademacher functionals and a multivariate normal distribution will be measured by the so-called $d_4$-distance that is defined as follows. Fix $d \in \NN$ and let $n=1, \dotsc, d$. For an $n$ times partially differentiable function $g: \RR^d \rightarrow \RR$ we put
\begin{align*}
M_k(g) := \max_{1 \leq i_1, \dotsc, i_k \leq d} \Bignorm{\frac{\partial^k}{\partial x_{i_1} \dotsc \partial x_{i_k}}g}_\infty
\end{align*}
for every $k=1, \dotsc, n$, where $\|\,\cdot\,\|_\infty$ denotes the supremum norm of the argument function. The $d_4$-distance between the distributions of two $\RR^d$-valued random vectors $\bX$ and $\bY$ is defined by
\begin{align*}
d_4(\bX, \bY) := \sup_{g} \absolute{\EE[g(\bX)] - \EE[g(\bY)]}\,,
\end{align*}
where the supremum is running over all four times partially differentiable functions $g: \RR^d \rightarrow \RR$ with bounded partial derivatives fulfilling $M_1(g), M_2(g), M_3(g), M_4(g) \leq 1$.

\begin{theorem}\label{thm:Multivariate main theorem}
Fix $d \in \NN$ and let $F_1, \dotsc, F_d$ be Rademacher functionals with $F_i \in \dom(D)$, $\EE[F_i] = 0$ and $\EE[\norm{(pq)^{-1/4} DF_i}_{\ell^4(\NN)}^4] < \infty$, for every $i=1, \dotsc, d$. Define $\bF := (F_1, \dotsc, F_d)$ and let $\bN:=(N_1, \dotsc, N_d)$ be a centred Gaussian random vector with symmetric and positive semidefinite covariance matrix $\Sigma:=(\Sigma_{ij})_{i,j=1}^d$. Further, let
\begin{align*}
A_1 &:= \frac{1}{2} \sum_{i,j=1}^d \EE[\absolute{\Sigma_{ij} - \langle DF_j, -DL^{-1}F_i \rangle_{\ell^2(\NN)}}]\,,\\
A_2 &:= \frac{1}{4} \EE \Big[ \Big\langle \frac{\absolute{p-q}}{\sqrt{pq}} \Big( \sum_{j=1}^d \absolute{DF_j} \Big)^2, \sum_{i=1}^d \absolute{-DL^{-1}F_i} \Big\rangle_{\ell^2(\NN)} \Big]\,,\\
A_3 &:= \frac{5}{24} \EE \Big[ \Big\langle \frac{1}{pq} \Big( \sum_{j=1}^d \absolute{DF_j} \Big)^3, \sum_{i=1}^d \absolute{-DL^{-1}F_i} \Big\rangle_{\ell^2(\NN)} \Big]\,.
\end{align*}
Then,
\begin{align*}
d_4(\bF,\bN) &\leq A_1+A_2+A_3\,.
\end{align*}
\end{theorem}

\begin{remark}\label{rem:NewTerm}\rm
A comparison of Theorem \ref{thm:Multivariate main theorem} with Theorem 5.1 in \cite{KRT1} shows that the extension to vectors of general functionals of possibly non-symmetric and non-homogeneous infinite Rademacher sequences comes at the costs of an additional summand in the bound, namely
\begin{align*}
\EE \Big[ \Big\langle \frac{\absolute{p-q}}{\sqrt{pq}} \Big( \sum_{j=1}^d \absolute{DF_j} \Big)^2, \sum_{i=1}^d \absolute{-DL^{-1}F_i} \Big\rangle_{\ell^2(\NN)} \Big]\,.
\end{align*}
However, resorting to the case where the underlying Rademacher sequence is symmetric, i.e., if $p_k=q_k=1/2$, for every $k \in \NN$, this additional summand vanishes and our bound in Theorem \ref{thm:Multivariate main theorem} coincides with the one from \cite{KRT1} with an improvement by a factor $1/2$ on the constant in front of the third term.
\end{remark}

The proof of Theorem \ref{thm:Multivariate main theorem} relies on two multivariate integration by parts formulae, a Gaussian one and an approximate one from Malliavin calculus which combines \eqref{Integration by parts formula} with a multivariate chain rule for the discrete gradient operator. We start by recalling the multivariate Gaussian integration by parts formula from Equation (A.41) in \cite{Tal}.

\begin{lemma}\label{Multivariate integration by parts lemma}
Fix $d \in \NN$ and let $\bN:=(N_1, \dotsc, N_d)$ be a centred Gaussian random vector with symmetric and positive semidefinite covariance matrix $\Sigma:=(\Sigma_{ij})_{i,j=1}^d$. Furthermore, let $g: \RR^d \rightarrow \RR$ be a partially differentiable function with bounded partial derivatives and $\EE[\absolute{N_i g(\bN)}] < \infty$, for every $i = 1, \dotsc, d$. Then, for every $i = 1, \dotsc, d$,
\begin{align*}
\EE[N_i g(\bN)] = \sum_{j=1}^d \Sigma_{ij} \EE \Big[ \frac{\partial}{\partial x_j} g(\bN) \Big]\,.
\end{align*}
\end{lemma}

The following lemma contains a multivariate chain rule for the discrete gradient operator, which is a generalization of Proposition 2.1 in \cite{PrivaultTorrisi} to the $d$-dimensional case. Also note that it not only generalizes Lemma 5.1 in \cite{KRT1} to the case where the underlying Rademacher sequence is non-symmetric and non-homogeneous, but also improves on the constants in the bound for the remainder term. For these reasons, we include a detailed proof. 

\begin{lemma}\label{Multivariate chain rule}
Let $\bF$ be a random vector of Rademacher functionals as in Theorem \ref{thm:Multivariate main theorem}. Furthermore, let $f: \RR^d \rightarrow \RR$ be a thrice partially differentiable function. Then, for every $k \in \NN$,
\begin{equation}\label{Multivariate chain rule equation 1}
\begin{split}
D_kf(\bF) &= \sum_{i=1}^d \frac{\partial}{\partial x_i} f(\bF)D_kF_i  - \frac{X_k}{4\sqrt{p_kq_k}} \sum_{i,j=1}^d \Big( \frac{\partial^2}{\partial x_i \partial x_j}f(\bF_k^+) + \frac{\partial^2}{\partial x_i \partial x_j}f(\bF_k^-) \Big)(D_kF_i)(D_kF_j) + R_k(\bF)\,
\end{split}
\end{equation}
with $\bF_k^\pm:=\big((F_1)_k^\pm,\ldots,(F_d)_k^\pm\big)$ and a remainder term $R_k(\bF)$ that fulfils
\begin{align}\label{Multivariate chain rule equation 2}
R_k(\bF) \leq \frac{5}{12p_kq_k} \sum_{i,j,\ell=1}^d \Bignorm{\frac{\partial^3}{\partial x_i \partial x_j \partial x_\ell}f}_\infty \absolute{(D_kF_i)(D_kF_j)(D_kF_\ell)}
\end{align}
for every $k \in \NN$.
\end{lemma}
\begin{proof}
Fix $k \in \NN$ and observe that
\begin{align}\label{Multivariate chain rule proof equation 1}
D_kf(\bF) &= \sqrt{p_kq_k}\, (f(\bF_k^+) - f(\bF_k^-)) = \sqrt{p_kq_k} \,(f(\bF_k^+) - f(\bF)) - \sqrt{p_kq_k}\, (f(\bF_k^-) - f(\bF))\,.
\end{align}
Now, a Taylor series expansion of $f$ at $\bF$ yields that, for every $\bm{x}:=(x_1, \dotsc, x_d) \in \RR^d$,
\begin{align*}
f(\bm{x}) - f(\bF) &= \sum_{i=1}^d \frac{\partial}{\partial x_i} f(\bF)(x_i-F_i) + \frac{1}{2}\sum_{i,j=1}^d \frac{\partial^2}{\partial x_i \partial x_j} f(\bF)(x_i-F_i)(x_j-F_j) \notag\\
&\phantom{{}={}} +\frac{1}{6} \sum_{i,j,\ell=1}^d \frac{\partial^3}{\partial x_i \partial x_j \partial x_\ell}f(\bF+\theta(\bm{x}-\bF)) (x_i-F_i)(x_j-F_j)(x_\ell-F_\ell)
\end{align*}
with some $\theta := \theta(\bm{x}, \bF) \in (0,1)$. By re-writing each of the quantities $f(\bF_k^+) - f(\bF)$ and $f(\bF_k^-) - f(\bF)$ in this way, it follows from \eqref{Multivariate chain rule proof equation 1} that, for every $k \in \NN$,
\begin{align}\label{Multivariate chain rule proof equation 2}
D_kf(\bF) &= \sqrt{p_kq_k} \sum_{i=1}^d \frac{\partial}{\partial x_i} f(\bF)((F_i)_k^+ - F_i) + \frac{1}{2}\sqrt{p_kq_k} \sum_{i,j=1}^d \frac{\partial^2}{\partial x_i \partial x_j} f(\bF)((F_i)_k^+ - F_i)((F_j)_k^+ - F_j) \notag\\
&\qquad + R_1(\bF, \bF_k^+) - \sqrt{p_kq_k} \sum_{i=1}^d \frac{\partial}{\partial x_i} f(\bF)((F_i)_k^- - F_i) \notag\\
&\qquad -\frac{1}{2}\sqrt{p_kq_k} \sum_{i,j=1}^d \frac{\partial^2}{\partial x_i \partial x_j} f(\bF)((F_i)_k^- - F_i)((F_j)_k^- - F_j) - R_2(\bF, \bF_k^-) \notag\\[5pt]
&= \sum_{i=1}^d \frac{\partial}{\partial x_i} f(\bF)D_kF_i +\frac{1}{2}\sqrt{p_kq_k} \sum_{i,j=1}^d \frac{\partial^2}{\partial x_i \partial x_j} f(\bF)(((F_i)_k^+ - F_i)((F_j)_k^+ - F_j) - ((F_i)_k^- - F_i)((F_j)_k^- - F_j)) \notag\\[5pt]
&\qquad+ R_1(\bF, \bF_k^+) - R_2(\bF, \bF_k^-)\,,
\end{align}
where from the identities
\begin{align}\label{F_k - F +}
F_k^+ - F = (F_k^+ - F_k^-)\1_{\{ X_k = -1 \}} = \frac{1}{\sqrt{p_kq_k}} (D_kF) \1_{\{ X_k = -1 \}}
\end{align}
and
\begin{align}\label{F_k - F -}
F_k^- - F = (F_k^- - F_k^+)\1_{\{ X_k = +1 \}} = -\frac{1}{\sqrt{p_kq_k}} (D_kF) \1_{\{ X_k = +1 \}}
\end{align}
it follows that, for every $k \in \NN$,
\begin{align}\label{Multivariate chain rule proof R_1}
\absolute{R_1(\bF, \bF_k^+)} &\leq \frac{1}{6}\sqrt{p_kq_k} \sum_{i,j,\ell=1}^d \Bignorm{\frac{\partial^3}{\partial x_i \partial x_j \partial x_\ell}f}_\infty \absolute{((F_i)_k^+ - F_i)((F_j)_k^+ - F_j)((F_\ell)_k^+ - F_\ell)} \notag\\
&= \frac{1}{6p_kq_k} \sum_{i,j,\ell=1}^d \Bignorm{\frac{\partial^3}{\partial x_i \partial x_j \partial x_\ell}f}_\infty \absolute{(D_kF_i)(D_kF_j)(D_kF_\ell)}\1_{\{ X_k=-1 \}}
\end{align}
and
\begin{align}\label{Multivariate chain rule proof R_2}
\absolute{R_2(\bF, \bF_k^-)} &\leq \frac{1}{6}\sqrt{p_kq_k} \sum_{i,j,\ell=1}^d \Bignorm{\frac{\partial^3}{\partial x_i \partial x_j \partial x_\ell}f}_\infty \absolute{((F_i)_k^- - F_i)((F_j)_k^- - F_j)((F_\ell)_k^- - F_\ell)} \notag\\
&= \frac{1}{6p_kq_k} \sum_{i,j,\ell=1}^d \Bignorm{\frac{\partial^3}{\partial x_i \partial x_j \partial x_\ell}f}_\infty \absolute{(D_kF_i)(D_kF_j)(D_kF_\ell)}\1_{\{ X_k=+1 \}}\,.
\end{align}
Again, by virtue of \eqref{F_k - F +} and \eqref{F_k - F -}, for every $k \in \NN$, the second summand on the right hand side of \eqref{Multivariate chain rule proof equation 2} can be rewritten as
\begin{align}\label{Multivariate chain rule proof equation 3}
&\frac{1}{2}\sqrt{p_kq_k} \sum_{i,j=1}^d \frac{\partial^2}{\partial x_i \partial x_j} f(\bF)(((F_i)_k^+ - F_i)((F_j)_k^+ - F_j) - ((F_i)_k^- - F_i)((F_j)_k^- - F_j)) \notag\\
&= \frac{1}{2\sqrt{p_kq_k}} \sum_{i,j=1}^d \frac{\partial^2}{\partial x_i \partial x_j} f(\bF)(D_kF_i)(D_kF_j)(\1_{\{ X_k = -1\}} - \1_{\{ X_k = +1\}}) \notag \\
&= -\frac{X_k}{2\sqrt{p_kq_k}} \sum_{i,j=1}^d \frac{\partial^2}{\partial x_i \partial x_j} f(\bF)(D_kF_i)(D_kF_j)\,.
\end{align}
Another Taylor series expansion of $\frac{\partial^2}{\partial x_i \partial x_j}f$ at $\bF_k^+$ and $\bF_k^-$, respectively, yields that, for every $i,j,k \in \NN$,
\begin{align*}
\frac{\partial^2}{\partial x_i \partial x_j}f(\bF) = \frac{\partial^2}{\partial x_i \partial x_j}f(\bF_k^+) + \sum_{\ell=1}^d \frac{\partial^3}{\partial x_\ell \partial x_i \partial x_j}f(\bF_k^+ + \theta_1(\bF - \bF_k^+))(F_\ell - (F_\ell)_k^+)\,\phantom{,}
\end{align*}
and
\begin{align*}
\frac{\partial^2}{\partial x_i \partial x_j}f(\bF) = \frac{\partial^2}{\partial x_i \partial x_j}f(\bF_k^-) + \sum_{\ell=1}^d \frac{\partial^3}{\partial x_\ell \partial x_i \partial x_j}f(\bF_k^- + \theta_2(\bF - \bF_k^-))(F_\ell - (F_\ell)_k^-)\,,
\end{align*}
where $\theta_1:=\theta_1(\bF, \bF_k^+)\in(0,1)$ and $\theta_2:=\theta_2(\bF, \bF_k^-) \in (0,1)$. This adds up to
\begin{align*}
\frac{\partial^2}{\partial x_i \partial x_j}f(\bF) &= \frac{1}{2}\Big( \frac{\partial^2}{\partial x_i \partial x_j}f(\bF_k^+) + \frac{\partial^2}{\partial x_i \partial x_j}f(\bF_k^-) \Big)+\frac{1}{2}\sum_{\ell=1}^d \frac{\partial^3}{\partial x_\ell \partial x_i \partial x_j}f(\bF_k^+ + \theta_1(\bF - \bF_k^+))(F_\ell - (F_\ell)_k^+)\\
&\qquad+\frac{1}{2}\sum_{\ell=1}^d \frac{\partial^3}{\partial x_\ell \partial x_i \partial x_j}f(\bF_k^- + \theta_2(\bF - \bF_k^-))(F_\ell - (F_\ell)_k^-)
\end{align*}
for every $i,j,k \in \NN$, and thus, it follows from \eqref{Multivariate chain rule proof equation 3} that for every $k \in \NN$
\begin{align}\label{Multivariate chain rule proof equation 4}
&\frac{1}{2}\sqrt{p_kq_k} \sum_{i,j=1}^d \frac{\partial^2}{\partial x_i \partial x_j} f(\bF)(((F_i)_k^+ - F_i)((F_j)_k^+ - F_j) - ((F_i)_k^- - F_i)((F_j)_k^- - F_j)) \notag\\
&= -\frac{X_k}{4\sqrt{p_kq_k}} \sum_{i,j=1}^d \Big( \frac{\partial^2}{\partial x_i \partial x_j}f(\bF_k^+) + \frac{\partial^2}{\partial x_i \partial x_j}f(\bF_k^-) \Big)(D_kF_i)(D_kF_j) - R_3(\bF,\bF_k^+) - R_4(\bF,\bF_k^-)\,,
\end{align}
where by the fact that $\absolute{X_k} \leq 1$ for every $k \in \NN$ and another application of \eqref{F_k - F +} and \eqref{F_k - F -} it holds that
\begin{align}\label{Multivariate chain rule proof R_3}
\absolute{R_3(\bF,\bF_k^+)} &\leq \frac{1}{4\sqrt{p_kq_k}} \sum_{i,j,\ell=1}^d \Bignorm{\frac{\partial^3}{\partial x_i \partial x_j \partial x_\ell}f}_\infty \absolute{(D_kF_j)(D_kF_\ell)((F_i)_k^+ - F_i)} \notag\\
&= \frac{1}{4p_kq_k} \sum_{i,j,\ell=1}^d \Bignorm{\frac{\partial^3}{\partial x_i \partial x_j \partial x_\ell}f}_\infty \absolute{(D_kF_i)(D_kF_j)(D_kF_\ell)}\1_{\{ X_k = -1 \}}
\end{align}
and
\begin{align}\label{Multivariate chain rule proof R_4}
\absolute{R_4(\bF,\bF_k^-)} &\leq \frac{1}{4\sqrt{p_kq_k}} \sum_{i,j,\ell=1}^d \Bignorm{\frac{\partial^3}{\partial x_i \partial x_j \partial x_\ell}f}_\infty \absolute{(D_kF_j)(D_kF_\ell)((F_i)_k^- - F_i)} \notag\\
&= \frac{1}{4p_kq_k} \sum_{i,j,\ell=1}^d \Bignorm{\frac{\partial^3}{\partial x_i \partial x_j \partial x_\ell}f}_\infty \absolute{(D_kF_i)(D_kF_j)(D_kF_\ell)}\1_{\{ X_k = +1 \}}\,.
\end{align}
Combining \eqref{Multivariate chain rule proof equation 2} and \eqref{Multivariate chain rule proof equation 4} finally yields that for every $k \in \NN$
\begin{align*}
D_kf(\bF) &= \sum_{i=1}^d \frac{\partial}{\partial x_i} f(\bF)D_kF_i - \frac{X_k}{4\sqrt{p_kq_k}} \sum_{i,j=1}^d \Big( \frac{\partial^2}{\partial x_i \partial x_j}f(\bF_k^+) + \frac{\partial^2}{\partial x_i \partial x_j}f(\bF_k^-) \Big)(D_kF_i)(D_kF_j) \notag\\[5pt]
&\qquad+ R_1(\bF, \bF_k^+) - R_2(\bF, \bF_k^-) - R_3(\bF,\bF_k^+) - R_4(\bF,\bF_k^-)\,,
\end{align*}
where because of \eqref{Multivariate chain rule proof R_1}, \eqref{Multivariate chain rule proof R_2}, \eqref{Multivariate chain rule proof R_3} and \eqref{Multivariate chain rule proof R_4} we have that
\begin{align*}
&\absolute{R_1(\bF, \bF_k^+)} + \absolute{R_2(\bF, \bF_k^-)} + \absolute{R_3(\bF,\bF_k^+)} + \absolute{R_4(\bF,\bF_k^-)} \leq\frac{5}{12p_kq_k} \sum_{i,j,\ell=1}^d \Bignorm{\frac{\partial^3}{\partial x_i \partial x_j \partial x_\ell}f}_\infty \absolute{(D_kF_i)(D_kF_j)(D_kF_\ell)}\,.
\end{align*}
The proof is thus complete.
\end{proof}

Let us now turn to the already announced multivariate approximate integration by parts formula. The next result not only generalizes Lemma 5.2 in \cite{KRT1} to the case in which the underlying Rademacher sequence is allowed to be non-symmetric and non-homogeneous, but also improves the constants in the bound for the remainder term. We emphasize that Lemma \ref{Multivariate approximate integration by parts} is the first instance where the additional boundary term discussed in Remark \ref{rem:NewTerm} shows up.

\begin{lemma}\label{Multivariate approximate integration by parts}
Let $\bF$ be a vector of Rademacher functionals as in Theorem \ref{thm:Multivariate main theorem}. Furthermore, let $f: \RR^d \rightarrow \RR$ be a thrice partially differentiable function with bounded partial derivatives. Then, for every $i=1, \dotsc, d$,
\begin{align*}
\EE[F_i f(\bF)] &= \sum_{j=1}^d \EE \Big[ \frac{\partial}{\partial x_j}f(\bF) \langle DF_j, -DL^{-1}F_i \rangle_{\ell^2(\NN)} \Big] + \EE[\langle R(\bF), -DL^{-1}F_i \rangle_{\ell^2(\NN)}]
\end{align*}
with a remainder $R(\bF)$ that satisfies the estimate
\begin{equation}\label{Multivariate approximate integration by parts equation 2}
\begin{split}
\absolute{\EE[\langle R(\bF), -DL^{-1}F_i \rangle_{\ell^2(\NN)}]} &\leq \frac{1}{2} M_2(f) \EE \Big[ \Big\langle \frac{\absolute{p-q}}{\sqrt{pq}} \Big( \sum_{j=1}^d \absolute{DF_j} \Big)^2, \absolute{-DL^{-1}F_i} \Big\rangle_{\ell^2(\NN)} \Big]\\
&\qquad+ \frac{5}{12} M_3(f) \EE \Big[ \Big\langle \frac{1}{pq} \Big( \sum_{j=1}^d \absolute{DF_j} \Big)^3, \absolute{-DL^{-1}F_i} \Big\rangle_{\ell^2(\NN)} \Big]\,.
\end{split}
\end{equation}
\end{lemma}
\begin{proof}
Fix $i=1, \dotsc, d$. By the integration by parts formula \eqref{Integration by parts formula} we have that
\begin{align}\label{Multivariate approximate integration by parts proof equation 1}
\EE[F_i f(\bF)] = \EE[ \langle Df(\bF), -DL^{-1}F_i \rangle_{\ell^2(\NN)}]\,.
\end{align}
Here, we implicitly used the fact that $f(\bF) \in \dom(D)$, which can be verified as follows. At first, by the mean value theorem it holds that for every $k \in \NN$
\begin{align*}
\absolute{D_kf(\bF)} &= \sqrt{p_kq_k}\,\absolute{f(\bF_k^+) - f(\bF_k^-)}= \sqrt{p_kq_k} \,\Bigabsolute{\sum_{i=1}^d \frac{\partial}{\partial x_i} f(\bF_k^- + \theta(\bF_k^+ - \bF_k^-))((F_i)_k^+ - (F_i)_k^-)}\\
&\leq \sqrt{p_kq_k}\, M_1(f) \sum_{i=1}^d \absolute{(F_i)_k^+ - (F_i)_k^-}= M_1(f) \sum_{i=1}^d \absolute{D_kF_i}\,,
\end{align*}
where $\theta \in (0,1)$. Thus, an application of the Cauchy-Schwarz inequality yields that
\begin{align}\label{Multivariate approximate integration by parts proof equation 2}
\EE[\norm{Df(\bF)}_{\ell^2(\NN)}^2] &= \EE \Big[ \sum_{k=1}^\infty (D_kf(\bF))^2 \Big] \leq (M_1(f))^2 \,\EE \Big[ \sum_{k=1}^\infty \Big( \sum_{i=1}^d \absolute{D_kF_i} \Big)^2 \Big] \leq d (M_1(f))^2 \,\EE \Big[ \sum_{k=1}^\infty \sum_{i=1}^d (D_kF_i)^2 \Big] \notag\\
&= d (M_1(f))^2 \sum_{i=1}^d \EE[\norm{DF_i}_{\ell^2(\NN)}^2] 
\end{align}
and finiteness of the right hand side in \eqref{Multivariate approximate integration by parts proof equation 2} follows from the assumptions that, for every $i=1, \dotsc, d$, $\frac{\partial}{\partial x_i}f$ is bounded and $F_i \in \dom(D)$. Now, by plugging \eqref{Multivariate chain rule equation 1} into \eqref{Multivariate approximate integration by parts proof equation 1} we immediately get
\begin{align}\label{Multivariate approximate integration by parts proof equation 3}
\EE[F_i f(\bF)] & = \sum_{j=1}^d \EE \Big[ \frac{\partial}{\partial x_j} f(\bF) \langle DF_j, -DL^{-1}F_i \rangle_{\ell^2(\NN)} \Big] \notag\\
&\qquad - \sum_{j,\ell=1}^d \EE \Big[ \Big\langle \frac{X}{4\sqrt{pq}} \Big( \frac{\partial^2}{\partial x_j \partial x_\ell}f(\bF^+) + \frac{\partial^2}{\partial x_j \partial x_\ell}f(\bF^-) \Big)(DF_j)(DF_\ell), -DL^{-1}F_i \Big\rangle_{\ell^2(\NN)} \Big] \notag\\[5pt]
&\qquad + \EE[\langle R_1(\bF), -DL^{-1}F_i \rangle_{\ell^2(\NN)}]
\end{align}
with $\bF^+ := (\bF_k^+)_{k \in \NN}$ and $\bF^- := (\bF_k^-)_{k \in \NN}$ as well as a remainder $R_1(\bF)$ which by \eqref{Multivariate chain rule equation 2} fulfils the estimate
\begin{align}\label{Multivariate approximate integration by parts proof remainder 1}
\absolute{\EE[\langle R_1(\bF), -DL^{-1}F_i \rangle_{\ell^2(\NN)}]} \leq \frac{5}{12} M_3(f) \EE \Big[ \Big\langle \frac{1}{pq} \Big( \sum_{j=1}^d \absolute{DF_j} \Big)^3, \absolute{-DL^{-1}F_i} \Big\rangle_{\ell^2(\NN)} \Big]\,.
\end{align}
As a consequence, we only need to further bound the second term in \eqref{Multivariate approximate integration by parts proof equation 3}. By virtue of the Cauchy-Schwarz inequality and \eqref{eq:MehlerInequality} we see that, for every $j,\ell \in \NN$,
\begin{align*}
&\EE \Big[ \sum_{k=1}^\infty \Bigabsolute{\frac{X_k}{4\sqrt{p_kq_k}} \Big( \frac{\partial^2}{\partial x_j \partial x_\ell}f(\bF_k^+) + \frac{\partial^2}{\partial x_j \partial x_\ell}f(\bF_k^-) \Big)(D_kF_j)(D_kF_\ell)(-D_kL^{-1}F_i)} \Big] \notag\\
&\leq \frac{1}{2} M_2(f) \EE \Big[ \sum_{k=1}^\infty \frac{1}{\sqrt{p_kq_k}} \absolute{(D_kF_j)(D_kF_\ell)(-D_kL^{-1}F_i)} \Big] \notag\\
&\leq \frac{1}{2} M_2(f) \Big( \EE \Big[ \sum_{k=1}^\infty \frac{1}{p_kq_k} (D_kF_j)^2(D_kF_\ell)^2 \Big] \Big)^{1/2} \Big( \EE \Big[ \sum_{k=1}^\infty (D_kL^{-1}F_i)^2 \Big] \Big)^{1/2} \notag\\
&\leq \frac{1}{2} M_2(f) \Big( \EE \Big[ \sum_{k=1}^\infty \frac{1}{p_kq_k} (D_kF_j)^4 \Big] \Big)^{1/4} \Big( \EE \Big[ \sum_{k=1}^\infty \frac{1}{p_kq_k} (D_kF_\ell)^4 \Big] \Big)^{1/4} \Big( \EE \Big[ \sum_{k=1}^\infty (D_kF_i)^2 \Big] \Big)^{1/2} \notag\\
&= \frac{1}{2} M_2(f) (\EE[\norm{(pq)^{-1/4} DF_j}_{\ell^4(\NN)}^4])^{1/4} (\EE[\norm{(pq)^{-1/4} DF_\ell}_{\ell^4(\NN)}^4])^{1/4} (\EE[\norm{DF_i}_{\ell^2(\NN)}^2])^{1/2} 
\end{align*}
and finiteness of this expression follows from the assumptions that $F_i \in \dom(D)$ and $\norm{(pq)^{-1/4} DF_i}_{\ell^2(\NN)} < \infty$ for every $i=1, \dotsc, d$. Thus, an exchange of expectation and summation is valid due to the Fubini-Tonelli theorem, and the independence of $X_k$ and $(\frac{\partial^2}{\partial x_j \partial x_\ell}f(\bF_k^+) + \frac{\partial^2}{\partial x_j \partial x_\ell}f(\bF_k^-))(D_kF_j)(D_kF_\ell)(-D_kL^{-1}F_i)$, for every $k \in \NN$, yields that, for every $j,\ell \in \NN$,
\begin{align}\label{Multivariate approximate integration by parts proof equation 5}
&\EE \Big[ \Big\langle \frac{X}{4\sqrt{pq}} \Big( \frac{\partial^2}{\partial x_j \partial x_\ell}f(\bF^+) + \frac{\partial^2}{\partial x_j \partial x_\ell}f(\bF^-) \Big)(DF_j)(DF_\ell), -DL^{-1}F_i \Big\rangle_{\ell^2(\NN)} \Big] \notag\\
&= \sum_{k=1}^\infty \frac{p_k-q_k}{4\sqrt{p_kq_k}} \EE \Big[ \Big( \frac{\partial^2}{\partial x_j \partial x_\ell}f(\bF_k^+) + \frac{\partial^2}{\partial x_j \partial x_\ell}f(\bF_k^-) \Big)(D_kF_j)(D_kF_\ell)(-D_kL^{-1}F_i) \Big] \notag \\
&= \EE \Big[ \Big\langle \frac{p-q}{4\sqrt{pq}} \Big( \frac{\partial^2}{\partial x_j \partial x_\ell}f(\bF^+) + \frac{\partial^2}{\partial x_j \partial x_\ell}f(\bF^-) \Big)(DF_j)(DF_\ell), -DL^{-1}F_i \Big\rangle_{\ell^2(\NN)} \Big]\,.
\end{align}
By plugging \eqref{Multivariate approximate integration by parts proof equation 5} into \eqref{Multivariate approximate integration by parts proof equation 3} we then get 
\begin{align}\label{Multivariate approximate integration by parts proof equation 6}
\EE[F_i f(\bF)] &= \sum_{j=1}^d \EE \Big[ \frac{\partial}{\partial x_j} f(\bF) \langle DF_j, -DL^{-1}F_i \rangle_{\ell^2(\NN)} \Big] + \EE[\langle R_1(\bF), -DL^{-1}F_i \rangle_{\ell^2(\NN)}]\notag\\
&\qquad\qquad - \EE[\langle R_2(\bF), -DL^{-1}F_i \rangle_{\ell^2(\NN)}]
\end{align}
with a remainder term
\begin{align*}
R_2(\bF) := \frac{p-q}{4\sqrt{pq}} \sum_{j,\ell=1}^\infty \Big( \frac{\partial^2}{\partial x_j \partial x_\ell}f(\bF^+) + \frac{\partial^2}{\partial x_j \partial x_\ell}f(\bF^-) \Big)(DF_j)(DF_\ell)
\end{align*}
satisfying
\begin{align}\label{Multivariate approximate integration by parts proof remainder 2}
&\absolute{\EE[\langle R_2(\bF), -DL^{-1}F_i \rangle_{\ell^2(\NN)}]} \leq \frac{1}{2} M_2(f) \EE \Big[ \Big\langle \frac{\absolute{p-q}}{\sqrt{pq}} \Big( \sum_{j=1}^d \absolute{DF_j} \Big)^2, \absolute{-DL^{-1}F_i} \Big\rangle_{\ell^2(\NN)} \Big]\,.
\end{align}
Finally, the assertion follows from \eqref{Multivariate approximate integration by parts proof equation 6} upon putting $R(\bF) := R_1(\bF) - R_2(\bF)$ and using the bounds in \eqref{Multivariate approximate integration by parts proof remainder 1} and \eqref{Multivariate approximate integration by parts proof remainder 2}. 
\end{proof}

\begin{remark}\rm
In the symmetric case where the underlying Rademacher sequence  satisfies $p_k=q_k=1/2$ for every $k \in \NN$, the bound for the remainder term in \eqref{Multivariate approximate integration by parts equation 2} simplifies to
\begin{align*}
\absolute{\EE[\langle R(\bF), -DL^{-1}F_i \rangle_{\ell^2(\NN)}]} &\leq \frac{5}{3} M_3(f) \EE \Big[ \Big\langle \Big( \sum_{j=1}^d \absolute{DF_j} \Big)^3, \absolute{-DL^{-1}F_i} \Big\rangle_{\ell^2(\NN)} \Big]\,,
\end{align*}
since $p_k-q_k=0$, for every $k \in \NN$.
\end{remark}

With both integration by parts formulae at hand we can now turn to the proof of Theorem \ref{thm:Multivariate main theorem}. We will use an interpolation technique that is known as the `smart path method' in the literature, cf.\ Section 2.4 in \cite{Tal}. This method has already found applications within the Malliavin-Stein method for multivariate normal approximation in the framework considering non-linear functionals of Gaussian \cite{NoudinPeccatiReveillac} and Poisson random measures \cite{PeccatiZheng}. We follow the lines of the proof of Theorem 4.2 in \cite{PeccatiZheng} until we have to use our discrete integration by parts formula developed in Lemma \ref{Multivariate approximate integration by parts}. Without loss of generality, we can and will from now on assume that $\bF$ and $\bN$ are independent.

\begin{proof}[Proof of Theorem \ref{thm:Multivariate main theorem}]
Let $g: \RR^d \to \RR$ be a four times partially differentiable function with bounded partial derivatives satisfying $M_1(g), M_2(g), M_3(g), M_4(g) \leq 1$. Consider the function $\Psi : \RR \rightarrow \RR$ given by
\begin{align}\label{Multivariate main theorem proof Psi}
\Psi(t) := \EE[g(\sqrt{1-t} \,\bF + \sqrt{t}\, \bN)]\,,\qquad t\in[0,1]\,.
\end{align}
Then, by the mean value theorem we have that
\begin{align}\label{Multivariate main theorem proof equation 1}
\absolute{\EE[g(\bF)] - \EE[g(\bN)]} = \absolute{\Psi(0) - \Psi(1)} \leq \sup_{t \in (0,1)} \absolute{\Psi'(t)}\,,
\end{align}
where, for every $t \in (0,1)$, $\Psi'$ is given by
\begin{align}\label{Multivariate main theorem proof equation 2}
\Psi'(t) &= \EE \Big[ \frac{d}{dt} g(\sqrt{1-t} \,\bF + \sqrt{t} \,\bN) \Big] = \sum_{i=1}^d \EE \Big[ \frac{\partial}{\partial x_i} g(\sqrt{1-t}\, \bF + \sqrt{t}\, \bN) \Big( \frac{1}{2\sqrt{t}}N_i - \frac{1}{2\sqrt{1-t}}F_i \Big) \Big] \notag\\
&= \frac{1}{2\sqrt{t}}A_t - \frac{1}{2\sqrt{1-t}}B_t
\end{align}
with
\begin{align*}
A_t := \sum_{i=1}^d \EE \Big[ \frac{\partial}{\partial x_i} g(\sqrt{1-t} \,\bF + \sqrt{t} \,\bN) N_i \Big] \quad \text{and} \quad B_t := \sum_{i=1}^d \EE \Big[ \frac{\partial}{\partial x_i} g(\sqrt{1-t} \,\bF + \sqrt{t} \,\bN) F_i \Big]\,.
\end{align*}

Now, by independence of $\bF$ and $\bN$ as well as by Fubini's theorem we have that, for every $t \in (0,1)$,
\begin{align}\label{Multivariate main theorem proof A_t 1}
A_t &= \sum_{i=1}^d \EE_ {\bN} \Big[ \EE_ {\bF} \Big[ \frac{\partial}{\partial x_i} g(\sqrt{1-t}\, \bF + \sqrt{t}\, \bN) \Big] N_i \Big]\,,
\end{align}
where $\EE_ {\bF}$ and $\EE_ {\bN}$ denote the expectations with respect to the distributions of $\bF$ and $\bN$, respectively. Using the functions $f_{i,t} : \R^d \rightarrow \R$ given by
\begin{align*}
f_{i,t}(\bm{x}) := \EE_ {\bF} \Big[ \frac{\partial}{\partial x_i} g(\sqrt{1-t}\, \bF + \sqrt{t}\, \bm{x}) \Big]\,,\qquad\bm{x}\in\RRd\,,
\end{align*}
\eqref{Multivariate main theorem proof A_t 1} can be rewritten as $A_t = \sum_{i=1}^d \EE_ {\bN}[f_{i,t}(\bN) N_i]$. Thus, by the integration by parts formula in Lemma \ref{Multivariate integration by parts lemma} we deduce that, for every $t \in (0,1)$,
\begin{align}\label{Multivariate main theorem proof A_t 3}
A_t &= \sum_{i,j=1}^d \Sigma_{ij} \EE_ {\bN} \Big[ \frac{\partial}{\partial x_j} f_{i,t}(\bN) \Big] = \sqrt{t} \sum_{i,j=1}^d \Sigma_{ij} \EE_ {\bN} \Big[ \EE_ {\bF} \Big[ \frac{\partial^2}{\partial x_j \partial x_i} g(\sqrt{1-t}\, \bF + \sqrt{t}\, \bN) \Big] \Big]\,,
\end{align}
where the exchange of differentiation and expectation in the last step of \eqref{Multivariate main theorem proof A_t 3} is valid since $\frac{\partial}{\partial x_i}g$ is bounded. Again, using the independence of $\bF$ and $\bN$ together with Fubini's theorem yields that, for every $t \in (0,1)$,
\begin{align}\label{Multivariate main theorem proof A_t 4}
A_t &= \sqrt{t} \sum_{i,j=1}^d \Sigma_{ij} \EE \Big[ \frac{\partial^2}{\partial x_j \partial x_i} g(\sqrt{1-t}\, \bF + \sqrt{t}\, \bN) \Big]\,.
\end{align}
Similarly as above, we deduce for the quantity $B_t$ that, for every $t \in (0,1)$,
\begin{align*}
B_t = \sum_{i=1}^d \EE_ {\bF} \Big[ \EE_ {\bN} \Big[ \frac{\partial}{\partial x_i} g(\sqrt{1-t}\, \bF + \sqrt{t}\, \bN) \Big] F_i \Big]\,.
\end{align*}
This time, defining the functions $h_{i,t} : \R^d \rightarrow \R$ by
\begin{align*}
h_{i,t}(\bm{x}) := \EE_ {\bN} \Big[ \frac{\partial}{\partial x_i} g(\sqrt{1-t}\, \bm{x} + \sqrt{t}\, \bN) \Big]\,,\qquad\bm{x}\in\RRd\,,
\end{align*}
and using the integration by parts formula in Lemma \ref{Multivariate approximate integration by parts}, we see that, for every $t \in (0,1)$,
\begin{align}\label{Multivariate main theorem proof B_t 1}
B_t = \sum_{i=1}^d \EE_ {\bF} \Big[ h_{i,t}(\bF) F_i \Big] &= \sum_{i,j=1}^d \EE_ {\bF} \Big[ \frac{\partial}{\partial x_j} h_{i,t}(\bF) \langle DF_j, -DL^{-1}F_i \rangle_{\ell^2(\NN)} \Big] + \sum_{i=1}^d \EE_ {\bF}[\langle R_{i,t}(\bF), -DL^{-1}F_i \rangle_{\ell^2(\NN)}] \notag\\
&= \sqrt{1-t} \sum_{i,j=1}^d \EE_ {\bF} \Big[ \EE_ {\bN} \Big[ \frac{\partial^2}{\partial x_j \partial x_i} g(\sqrt{1-t}\, \bF + \sqrt{t}\, \bN) \Big] \langle DF_j, -DL^{-1}F_i \rangle_{\ell^2(\NN)} \Big] \notag\\
&\qquad+ \sum_{i=1}^d \EE_ {\bF}[\langle R_{i,t}(\bF), -DL^{-1}F_i \rangle_{\ell^2(\NN)}]\,,
\end{align}
where we recall that the exchange of differentiation and expectation in the last step is valid since $\frac{\partial}{\partial x_i}g$ is bounded for every $i=1, \dots, d$, and $R_{i,t}(\bF)$ is a remainder which by \eqref{Multivariate approximate integration by parts equation 2} fulfils that for every $i=1, \dots, d$ and $t \in (0,1)$,
\begin{align}\label{Multivariate main theorem proof remainder}
\absolute{\EE[\langle R_{i,t}(\bF), -DL^{-1}F_i \rangle_{\ell(\NN)}]} 
&\leq \frac{1}{2} M_2(h_{i,t}) \EE \Big[ \Big\langle \frac{\absolute{p-q}}{\sqrt{pq}} \Big( \sum_{j=1}^d \absolute{DF_j} \Big)^2, \absolute{-DL^{-1}F_i} \Big\rangle_{\ell^2(\NN)} \Big] \notag\\
&\qquad + \frac{5}{12} M_3(h_{i,t}) \EE \Big[ \Big\langle \frac{1}{pq} \Big( \sum_{j=1}^d \absolute{DF_j} \Big)^3, \absolute{-DL^{-1}F_i} \Big\rangle_{\ell^2(\NN)} \Big] \notag\\
&\leq \frac{1}{2} (1-t)M_3(g) \EE \Big[ \Big\langle \frac{\absolute{p-q}}{\sqrt{pq}} \Big( \sum_{j=1}^d \absolute{DF_j} \Big)^2, \absolute{-DL^{-1}F_i} \Big\rangle_{\ell^2(\NN)} \Big] \notag\\
&\qquad + \frac{5}{12} (1-t)^{3/2} M_4(g) \EE \Big[ \Big\langle \frac{1}{pq} \Big( \sum_{j=1}^d \absolute{DF_j} \Big)^3, \absolute{-DL^{-1}F_i} \Big\rangle_{\ell^2(\NN)} \Big]\,.
\end{align}
Going back to \eqref{Multivariate main theorem proof B_t 1}, another application of the independence of $\bF$ and $\bN$ together with Fubini's theorem yields that for every $t \in (0,1)$,
\begin{align}\label{Multivariate main theorem proof B_t 2}
B_t &= \sqrt{1-t} \sum_{i,j=1}^d \EE_ {\bF} \Big[ \EE_ {\bN} \Big[ \frac{\partial^2}{\partial x_j \partial x_i} g(\sqrt{1-t}\, \bF + \sqrt{t}\, \bN) \langle DF_j, -DL^{-1}F_i \rangle_{\ell^2(\NN)} \Big] \Big] \notag\\
&\qquad + \sum_{i=1}^d \EE_ {\bF}[ \EE_ {\bN}[ \langle R_{i,t}(\bF), -DL^{-1}F_i \rangle_{\ell^2(\NN)}]] \notag \\
&= \sqrt{1-t} \sum_{i,j=1}^d \EE \Big[ \frac{\partial^2}{\partial x_j \partial x_i} g(\sqrt{1-t}\, \bF + \sqrt{t}\, \bN) \langle DF_j, -DL^{-1}F_i \rangle_{\ell^2(\NN)} \Big] \notag\\
&\qquad + \sum_{i=1}^d \EE[\langle R_{i,t}(\bF), -DL^{-1}F_i \rangle_{\ell^2(\NN)}]\,.
\end{align}
Hence, by combining \eqref{Multivariate main theorem proof A_t 4} and \eqref{Multivariate main theorem proof B_t 2} with \eqref{Multivariate main theorem proof equation 2} we deduce that for every $t \in (0,1)$,
\begin{align*}
\Psi'(t) &= \frac{1}{2} \sum_{i,j=1}^d \EE \Big[ \frac{\partial^2}{\partial x_j \partial x_i} g(\sqrt{1-t}\, \bF + \sqrt{t}\, \bN) (\Sigma_{ij} - \langle DF_j, -DL^{-1}F_i \rangle_{\ell^2(\NN)}) \Big] \notag\\
&\qquad - \frac{1}{2\sqrt{1-t}} \sum_{i=1}^d \EE[\langle R_{i,t}(\bF), -DL^{-1}F_i \rangle_{\ell^2(\NN)}]\,,
\end{align*}
and by using \eqref{Multivariate main theorem proof remainder} as well as the fact that $M_2(g), M_3(g), M_4(g) \leq 1$ we thus conclude that, for every $t \in (0,1)$,
\begin{align*}
\absolute{\Psi'(t)} &\leq \frac{1}{2} \sum_{i,j=1}^d \EE[\absolute{\Sigma_{ij} - \langle DF_j, -DL^{-1}F_i \rangle_{\ell^2(\NN)}}] \notag\\
&\phantom{{}={}} + \frac{1}{4}\sqrt{1-t} \EE \Big[ \Big\langle \frac{\absolute{p-q}}{\sqrt{pq}} \Big( \sum_{j=1}^d \absolute{DF_j} \Big)^2, \sum_{i=1}^d \absolute{-DL^{-1}F_i} \Big\rangle_{\ell^2(\NN)} \Big] \notag\\
&\phantom{{}\leq{}} + \frac{5}{24} (1-t) \EE \Big[ \Big\langle \frac{1}{pq} \Big( \sum_{j=1}^d \absolute{DF_j} \Big)^3, \sum_{i=1}^d \absolute{-DL^{-1}F_i} \Big\rangle_{\ell^2(\NN)} \Big].
\end{align*}
Plugging this into \eqref{Multivariate main theorem proof equation 1} where we take the supremum over all $t \in (0,1)$ completes the argument.
\end{proof}

\subsection{A multivariate discrete second-order Poincar\'e inequality}\label{sec:2ndOrderPoincare}

In this section we use Theorem \ref{thm:Multivariate main theorem} to develop a discrete second-order Poincar\'e inequality for the normal approximation of vectors of Rademacher functionals. In comparison with Theorem \ref{thm:Multivariate main theorem} it has the advantage that it expresses the bound for $d_4(\bF,\bN)$ only in terms of discrete first- and second-order Malliavin derivatives and does not involve the operator $L^{-1}$. This in turn allows to apply the bound without specifying the chaos decomposition of the component random variables of the random vector $\bF$. Our result can be seen as the natural multivariate extension of the main result from \cite{KRT2}, where a univariate discrete second-order Poincar\'e inequality has been obtained for the Kolmogorov distance (see also Remark 3.2 \cite{Zheng} for a closely related bound for the Wasserstein distance). 

\begin{theorem}\label{thm:SecondOrderPoincare}
Let the conditions of Theorem \ref{thm:Multivariate main theorem} prevail and assume additionally that $F_i\in\dom(D^2)$ for all $i=1,\ldots,d$. For $i,j=1,\ldots,d$ define
\begin{align*}
B_1(i,j) &:= \Big( \frac{15}{4} \sum_{k,\ell,m=1}^\infty (\EE[(D_kF_i)^2(D_\ell F_i)^2])^{1/2} (\EE[(D_mD_kF_j)^2(D_mD_\ell F_j)^2])^{1/2} \Big)^{1/2}\,,\\
B_2(i,j) &:= \Big( \frac{3}{4} \sum_{k,\ell,m=1}^\infty \frac{1}{p_mq_m} (\EE[(D_mD_kF_i)^2(D_mD_\ell F_i)^2])^{1/2} (\EE[(D_mD_kF_j)^2(D_mD_\ell F_j)^2])^{1/2}\Big)^{1/2}\,,\\
B_3(i,j) &:= \frac{1}{2} d^{3/2} \sum_{k=1}^\infty \frac{|p_k-q_k|}{\sqrt{p_kq_k}} (\EE[(D_kF_i)^2])^{1/2} (\EE[(D_kF_j)^4])^{1/2}\,,\\
B_4(i,j) &:= \frac{5}{12} d^2 \sum_{k=1}^\infty \frac{1}{p_kq_k} (\EE[(D_kF_i)^4])^{1/4} (\EE[(D_kF_j)^4])^{3/4}\,.
\end{align*}
Then,
\begin{align*}
d_4(\bF,\bN) \leq \frac{1}{2} \sum_{i,j=1}^d \big[ |\Sigma_{ij}-\cov(F_i,F_j)| + B_1(i,j) + B_2(i,j) + B_3(i,j) + B_4(i,j) \big]\,.
\end{align*}
\end{theorem}
\begin{proof}
Let $A_1,A_2$ and $A_3$ be the three terms defined in Theorem \ref{thm:Multivariate main theorem}. We start with $A_1$. An application of the triangle inequality yields
\begin{align}\label{eq:SecondOrderPoincareEq1}
\EE[\absolute{\Sigma_{ij} - \langle DF_j, -DL^{-1}F_i \rangle_{\ell^2(\NN)}}] \leq \absolute{\Sigma_{ij} - \cov(F_i,F_j)} + \EE[\absolute{\cov(F_i,F_j) - \langle DF_j, -DL^{-1}F_i \rangle_{\ell^2(\NN)}}].
\end{align}
Let us further consider the second summand on the right hand side of \eqref{eq:SecondOrderPoincareEq1}. By the integration by parts formula in \eqref{Integration by parts formula} we see that
\begin{align*}
\cov(F_i,F_j) = \EE[\langle DF_j,-DL^{-1}F_i\rangle_{\ell^2(\NN)}],
\end{align*}
and thus, by the Cauchy-Schwarz inequality and the Poincar\'e inequality in \eqref{eq:PoincareInequality} we have that
\begin{align}\label{eq:SecondOrderPoincareEq2}
\EE[\absolute{\cov(F_i,F_j) - \langle DF_j, -DL^{-1}F_i \rangle_{\ell^2(\NN)}}] &\leq (\Var(\langle DF_j,-DL^{-1}F_i\rangle_{\ell^2(\NN)}))^{1/2}\notag\\
&\leq (\EE[\norm{D(\langle DF_j,-DL^{-1}F_i\rangle_{\ell^2(\NN)})}_{\ell^2(\NN)}^2])^{1/2}\notag\\
&= \Big( \EE\Big[\sum_{\ell=1}^\infty \Big(\sum_{k=1}^\infty D_\ell ((D_kF_j)(-D_kL^{-1}F_i))\Big)^2 \Big] \Big)^{1/2}.
\end{align}
By the product formula for the discrete Malliavin derivative in \eqref{eq:ProductFormula} and the triangle inequality we get that, for every $k,\ell \in \NN$,
\begin{align*}
&\absolute{D_\ell ((D_kF_j)(-D_kL^{-1}F_i))}\\
&\leq \absolute{(D_\ell D_kF_j)(-D_kL^{-1}F_i)} + \absolute{(D_kF_j)(-D_\ell D_kL^{-1}F_i)} + \frac{1}{\sqrt{p_\ell q_\ell}} \absolute{(D_\ell D_kF_j)(-D_\ell D_kL^{-1}F_i)}\,.
\end{align*}
Plugging this into \eqref{eq:SecondOrderPoincareEq2} and using the Cauchy-Schwarz inequality then yields
\begin{align}\label{eq:SecondOrderPoincareEq3}
\EE[\absolute{\cov(F_i,F_j) - \langle DF_j, -DL^{-1}F_i \rangle_{\ell^2(\NN)}}] \leq \big(3\big[T_1(i,j)+T_2(i,j)+T_3(i,j)\big]\big)^{1/2}
\end{align}
with
\begin{align*}
T_1(i,j) &:= \EE\Big[\sum_{\ell=1}^\infty \Big(\sum_{k=1}^\infty \absolute{(D_\ell D_kF_j)(-D_kL^{-1}F_i)}\Big)^2 \Big]\,,\\
T_2(i,j) &:= \EE\Big[\sum_{\ell=1}^\infty \Big(\sum_{k=1}^\infty \absolute{(D_kF_j)(-D_\ell D_kL^{-1}F_i)}\Big)^2 \Big]\,,\\
T_3(i,j) &:= \EE\Big[\sum_{\ell=1}^\infty \frac{1}{p_\ell q_\ell} \Big(\sum_{k=1}^\infty \absolute{(D_\ell D_kF_j)(-D_\ell D_kL^{-1}F_i)}\Big)^2 \Big]\,.
\end{align*}
Each of these quantities will now be further bounded from above. Considering $T_1$, an application of \eqref{eq:IntegratedMehler} and \eqref{eq:Mehler} as well as the triangle inequality yields that, for every $\ell \in \NN$,
\begin{align*}
\Big(\sum_{k=1}^\infty \absolute{(D_\ell D_kF_j)(-D_kL^{-1}F_i)}\Big)^2 &= \Big(\sum_{k=1}^\infty \Bigabsolute{(D_\ell D_kF_j) \int_0^\infty e^{-t}P_tD_kF_i \, dt} \Big)^2\\
&= \Big(\sum_{k=1}^\infty \Bigabsolute{(D_\ell D_kF_j) \int_0^\infty e^{-t}\EE[D_kF_i(X^t) \, \vert \, X] \, dt} \Big)^2\\
&\leq \Big(\sum_{k=1}^\infty \absolute{D_\ell D_kF_j} \int_0^\infty e^{-t}\EE[\absolute{D_kF_i(X^t)} \, \vert \, X] \, dt \Big)^2\,.
\end{align*}
Furthermore, by virtue of the monotone convergence theorem we get that, for every $\ell \in \NN$,
\begin{align*}
\Big(\sum_{k=1}^\infty \absolute{D_\ell D_kF_j} \int_0^\infty e^{-t}\EE[\absolute{D_kF_i(X^t)} \, \vert \, X] \, dt \Big)^2 &= \Big(\int_0^\infty e^{-t} \sum_{k=1}^\infty \absolute{D_\ell D_kF_j} \EE[\absolute{D_kF_i(X^t)} \, \vert \, X] \, dt \Big)^2\\
&= \Big(\int_0^\infty e^{-t} \EE\Big[ \sum_{k=1}^\infty \absolute{(D_\ell D_kF_j)(D_kF_i(X^t))} \, \Big\vert \, X\Big] \,  dt \Big)^2\,.
\end{align*}
By using Jensen's inequality as well as the Cauchy-Schwarz inequality we then conclude that, for every $\ell \in \NN$,
\begin{align*}
&\Big(\int_0^\infty e^{-t} \EE\Big[ \sum_{k=1}^\infty \absolute{(D_\ell D_kF_j)(D_kF_i(X^t))} \, \Big\vert \, X\Big] \, dt \Big)^2\\
&\leq \int_0^\infty e^{-t} \EE\Big[ \Big( \sum_{k=1}^\infty \absolute{(D_\ell D_kF_j)(D_kF_i(X^t))} \Big)^2 \, \Big\vert \, X\Big] \, dt\\
&= \int_0^\infty e^{-t} \EE\Big[ \sum_{m,k=1}^\infty \absolute{(D_\ell D_mF_j)(D_mF_i(X^t))(D_\ell D_kF_j)(D_kF_i(X^t))} \, \Big\vert \, X\Big] \, dt\\
&= \sum_{m,k=1}^\infty \absolute{(D_\ell D_mF_j)(D_\ell D_kF_j)} \int_0^\infty e^{-t} \EE[ \absolute{(D_mF_i(X^t))(D_kF_i(X^t))} \, \vert \, X] \, dt\\
&\leq \sum_{m,k=1}^\infty \absolute{(D_\ell D_mF_j)(D_\ell D_kF_j)} \int_0^\infty e^{-t} (\EE[ (D_mF_i(X^t))^2 (D_kF_i(X^t))^2 \, \vert \, X])^{1/2} \, dt\\
&\leq \sum_{m,k=1}^\infty \absolute{(D_\ell D_mF_j)(D_\ell D_kF_j)} \Big( \int_0^\infty e^{-t} \EE[ (D_mF_i(X^t))^2 (D_kF_i(X^t))^2 \, \vert \, X] \, dt \Big)^{1/2}\,.
\end{align*}
Thus, another application of the Cauchy-Schwarz inequality leads to the bound
\begin{align*}
T_1(i,j) &\leq \EE \Big[ \sum_{m,k,\ell=1}^\infty \absolute{(D_\ell D_mF_j)(D_\ell D_kF_j)} \Big( \int_0^\infty e^{-t} \EE[ (D_mF_i(X^t))^2 (D_kF_i(X^t))^2 \, \vert \, X] \, dt \Big)^{1/2} \Big] \notag\\
&\leq \sum_{m,k,\ell=1}^\infty (\EE[(D_\ell D_mF_j)^2(D_\ell D_kF_j)^2])^{1/2} \Big( \EE \Big[ \int_0^\infty e^{-t} \EE[ (D_mF_i(X^t))^2 (D_kF_i(X^t))^2 \, \vert \, X] \, dt \Big] \Big)^{1/2} \notag\\
&= \sum_{m,k,\ell=1}^\infty (\EE[(D_\ell D_mF_j)^2(D_\ell D_kF_j)^2])^{1/2} \Big( \int_0^\infty e^{-t} \EE[(D_mF_i)^2(D_kF_i)^2] \, dt \Big)^{1/2} \notag\\
&= \sum_{m,k,\ell=1}^\infty (\EE[(D_\ell D_mF_j)^2(D_\ell D_kF_j)^2])^{1/2}(\EE[(D_mF_i)^2(D_kF_i)^2])^{1/2}\,.
\end{align*}
The quantities $T_2$ and $T_3$ can be treated in the same manner as $T_1$, and thus, it holds that
\begin{align*}
T_2(i,j) &\leq \frac{1}{4} \sum_{m,k,\ell=1}^\infty (\EE[(D_\ell D_mF_i)^2(D_\ell D_kF_i)^2])^{1/2}(\EE[(D_mF_j)^2(D_kF_j)^2])^{1/2},\\
T_3(i,j) &\leq \frac{1}{4} \sum_{m,k,\ell=1}^\infty \frac{1}{p_\ell q_\ell} (\EE[(D_\ell D_mF_i)^2(D_\ell D_kF_i)^2])^{1/2}(\EE[(D_\ell D_mF_j)^2(D_\ell D_kF_j)^2])^{1/2}\,.
\end{align*}
Therefore, combining the bounds for $T_1, T_2$ and $T_3$ with \eqref{eq:SecondOrderPoincareEq3} and \eqref{eq:SecondOrderPoincareEq1} yields
\begin{align*}
A_1 &\leq \frac{1}{2} \sum_{i,j=1}^d \big[ |\Sigma_{ij}-\cov(F_i,F_j)| + B_1(i,j) + B_2(i,j)\big]\,.
\end{align*}

Turning to the term $A_2$, by several applications of the Cauchy-Schwarz inequality and due to \eqref{eq:MehlerInequality} we see that
\begin{align*}
A_2 &= \frac{1}{4} \sum_{k=1}^\infty \frac{\absolute{p_k-q_k}}{\sqrt{p_kq_k}} \sum_{i=1}^d \EE \Big[ \Big( \sum_{j=1}^d \absolute{D_kF_j} \Big)^2 \absolute{-D_kL^{-1}F_i} \Big]\\
&\leq \frac{1}{4} \sum_{k=1}^\infty \frac{\absolute{p_k-q_k}}{\sqrt{p_kq_k}} \Big( \EE \Big[ \Big( \sum_{j=1}^d \absolute{D_kF_j} \Big)^4 \Big] \Big)^{1/2} \sum_{i=1}^d (\EE[(D_kF_i)^2])^{1/2}\\
&\leq \frac{1}{4}d^{3/2} \sum_{i,j=1}^d \sum_{k=1}^\infty \frac{\absolute{p_k-q_k}}{\sqrt{p_kq_k}} (\EE[(D_kF_j)^4])^{1/2} (\EE[(D_kF_i)^2])^{1/2}\,.
\end{align*}

For the third and last term $A_3$ we similarly see that, by using H\"older's inequality with H\"older conjugates $4$ and $4/3$ as well as \eqref{eq:MehlerInequality},
\begin{align*}
A_3 &= \frac{5}{24} \sum_{k=1}^\infty \frac{1}{p_kq_k} \sum_{i=1}^d \EE \Big[ \Big( \sum_{j=1}^d \absolute{D_kF_j} \Big)^3 \absolute{-D_kL^{-1}F_i} \Big]\\
&\leq \frac{5}{24}d^2 \sum_{i,j=1}^d \sum_{k=1}^\infty \frac{1}{p_kq_k} \EE[\absolute{D_kF_j}^3 \cdot \absolute{-D_kL^{-1}F_i}]\\
&\leq \frac{5}{24}d^2 \sum_{i,j=1}^d \sum_{k=1}^\infty \frac{1}{p_kq_k} (\EE[(D_kF_j)^4])^{3/4} (\EE[(D_kL^{-1}F_i)^4])^{1/4}\\
&\leq \frac{5}{24}d^2 \sum_{i,j=1}^d \sum_{k=1}^\infty \frac{1}{p_kq_k} (\EE[(D_kF_j)^4])^{3/4} (\EE[(D_kF_i)^4])^{1/4}\,.
\end{align*}
This completes the argument.
\end{proof}

\section{Applications to random graphs and random cubical complexes}\label{sec:Application}

\subsection{Subgraph and degree counts in the Erd\H{o}s-R\'enyi random graph}

In this section we consider an application of the discrete second-order Poincar\'e inequality developed above to subgraph counts in the Erd\H{o}s-R\'enyi random graph. To describe the model, let $K_n$ be the complete graph on $n\in\NN$ vertices and fix $p\in(0,1)$. In what follows we implicitly assume that $n$ is sufficiently large. We number the $n\choose 2$ edges of $K_n$ in a fixed but arbitrary way and denote them by $e_1,\ldots,e_{n\choose 2}$. Now, to each edge $e_k$ of $K_n$ a Rademacher random variable $X_k$ with success probability $p$ is assigned and we remove $e_k$ from $K_n$ if $X_k=-1$ and keep $e_k$ otherwise. This gives rise to the Erd\H{o}s-R\'enyi random graph denoted by $\cG(n,p)$, which has $n$ vertices and a binomially distributed number of edges with parameters ${n\choose 2}$ and $p$. In what follows we assume $p$ to be independent of $n$.

Let $\G$ be a fixed (finite, simple) graph and denote by $X_\G$ the number of subgraphs of $\cG(n,p)$ that are isomorphic to $\G$ (we assume here that all graphs we consider have at least one edge). To represent this counting statistic formally, we denote by $v_\G$ the number of vertices and by $e_\G$ the number of edges of $\G$. Moreover, we shall denote by $\aut(\G)$ the (finite) group of graph-automorphisms of $\G$ and by $|\aut(\G)|$ its cardinality. Using this notation, $X_\G$ may be written as 
\begin{equation}\label{eq:XGamma}
X_\G = \sum_{\G'}{\bf 1}\{\G'\subset\cG(n,p)\}\,,
\end{equation}
where the sum is running over all $\binom{n}{v_\G}{v_\G!\over|\aut(\G)|}$ copies $\G'$ of $\G$ in $K_n$ and where ${\bf 1}\{\G'\subset\cG(n,p)\}$ is the indicator function of the event that $\G'$ is a subgraph of $\cG(n,p)$. Since $\EE[{\bf 1}\{\G'\subset\cG(n,p)\}]=P(\G'\subset\cG(n,p))=p^{e_\G}$, it readily follows that
$$
\EE[X_\G] = {n\choose v_\G}{v_\G!\over|\aut(\G)|}p^{e_\G}\,.
$$
To proceed, we also need information about the covariance between $X_\G$ and $X_\Phi$ for two graphs $\G$ and $\Phi$. Before we state the result, let us introduce our asymptotic notation. We shall write $a_n=O(b_n)$ for two sequences $(a_n)_{n\in\NN}$ and $(b_n)_{n\in\NN}$ if $\limsup_{n\to\infty}|a_n/b_n|<\infty$. Moreover, $a_n\asymp b_n$ will indicate that $|a_n/b_n|\to 1$, as $n\to\infty
$.

\begin{lemma}\label{lem:CovarianceERG}
Let $\G$ and $\Phi$ be two graphs and define $X_\G$ and $X_\Phi$ as above. Then,
\begin{align*}
\cov(X_\G,X_\Phi) &\asymp 2\,{n^{v_\G+v_\Phi-2}\over|\aut(\G)||\aut(\Phi)|}\,e_\G e_\Phi\,p^{e_\G+e_\Phi-1}(1-p)\\
&\qquad+c(\G,\Phi)\,n^{v_\G+v_\Phi-3}\,p^{e_\G+e_\Phi-2}(1-p)+O(n^{v_\G+v_\Phi-3})
\end{align*}
with a constant $c(\G,\Phi)>0$ only depending on $\G$ and $\Phi$.
\end{lemma}
\begin{proof}
Recalling \eqref{eq:XGamma}, we see that
$$
\cov(X_\G,X_\Phi) = \sum_{\G',\Phi'}\cov({\bf 1}\{\G'\subset\cG(n,p)\},{\bf 1}\{\Phi'\subset\cG(n,p)\})\,.
$$
By the independence properties of the construction of $\cG(n,p)$ we have that $\cov({\bf 1}\{\G'\subset\cG(n,p)\},{\bf 1}\{\Phi'\subset\cG(n,p)\})\neq 0$ if and only if $\G'$ and $\Phi'$ have at least one common edge. In what follows we shall write $e_{\G',\Phi'}$ for the number of edges that $\G'$ and $\Phi'$ have in common. Thus,
\begin{align*}
\cov(X_\G,X_\Phi) &= \sum_{\G',\Phi':e_{\G',\Phi'}\geq 1}\cov({\bf 1}\{\G'\subset\cG(n,p)\},{\bf 1}\{\Phi'\subset\cG(n,p)\})\\
&=\sum_{\G',\Phi':e_{\G',\Phi'}\geq 1}\big(\EE[{\bf 1}\{\G'\subset\cG(n,p),\Phi'\subset\cG(n,p)\}]-\EE[{\bf 1}\{\G'\subset\cG(n,p)\}]\EE[{\bf 1}\{\Phi'\subset\cG(n,p)\}]\big)\\
&=\sum_{\G',\Phi':e_{\G',\Phi'}\geq 1}\big(p^{e_\G+e_\Phi-e_{\G',\Phi'}}-p^{e_\G}p^{e_\Phi}\big)\\
&=\sum_{\G',\Phi':e_{\G',\Phi'}\geq 1}p^{e_\G+e_\Phi-e_{\G',\Phi'}}(1-p^{e_{\G',\Phi'}})\\
&=\sum_{i=1}^{\min(e_\G,e_\Phi)}\sum_{\G'}\;\;\sum_{\Phi':e_{\G',\Phi'}=i}p^{e_\G+e_\Phi-i}(1-p^i)\,.
\end{align*}
Now, we notice that the second sum is running over $\binom{n}{v_\G}{v_\G!\over|\aut(\G)|}\asymp{n^{v_\G}\over|\aut(\G)|}$ terms. By choosing $i=1$ in the first sum (a choice that leads to the asymptotically dominating term), we see that the third sum is running over $\asymp{n^{v_\Phi-2}\over|\aut(\Phi)|}$ terms, since $\G'$ and $\Phi'$ have precisely one edge in common and there are ${n-v_\G\choose v_\Phi-2}\asymp n^{v_\Phi-2}$ possible choices for the $e_\Phi-1$ missing vertices to build a copy $\Phi'$ of $\Phi$ in $\cG(n,p)$. Moreover, taking into account all possible choices and orientations for this common edge gives rise to another factor $2e_\G e_\Phi$. Summarizing, the term with $i=1$ yields the asymptotic contribution 
$$
2\,{n^{v_\G+v_\Phi-2}\over|\aut(\G)||\aut(\Phi)|}\,e_\G e_\Phi\,p^{e_\G+e_\Phi-1}(1-p)\,.
$$
Choosing $i=2$ we see that there are two possible situations. Namely, the two common edges of $\G'$ and $\Phi'$ can or cannot have a common vertex. In the first situation and by the same reasoning as above, the asymptotic contribution is $\asymp c_1(\G,\Phi)n^{v_\G+v_\Phi-3}p^{e_\G+e_\Phi-2}(1-p^2)$, while in the second case we have the asymptotic contribution $\asymp c_2(\G,\Phi)n^{v_\G+v_\Phi-4}p^{e_\G+e_\Phi-2}(1-p^2)$ with constants $c_1(\G,\Phi),c_2(\G,\Phi)>0$ only depending on $\G$ and $\Phi$. Moreover, it is clear from this discussion that for all $i\geq 3$ the corresponding terms in the above sum are of order $O(n^{v_\G+v_\Phi-3})$. This proves the claim.
\end{proof}

Now, let us turn to the multivariate central limit theorem for the subgraph counting statistics $X_\G$. For this, fix some $d\in\NN$ and let $\G_1,\ldots,\G_d$ be $d$ fixed (finite, simple) graphs with associated counting statistics $X_{\G_1},\ldots,X_{\G_d}$. For $i\in\{1,\ldots,d\}$ define the normalized random variables $\widetilde{X}_{\G_i}:=n^{1-v_{\G_i}}(X_{\G_i}-\EE[X_{\G_i}])$ and the random vector $\bX_{\bfG}:=(\widetilde{X}_{\G_1},\ldots,\widetilde{X}_{\G_d})$. Our next result is the announced multivariate central limit theorem for $\bX_\bfG$, which adds a rate of convergence to the related result in the paper of Janson and Nowicki \cite{JansonNowicki}.

\begin{theorem}\label{thm:CLTRandomGraphs}
Let $\Sig=(\Sig_{ij})_{i,j=1}^d$ be the matrix given by
$$
\Sig_{ij}:=\sigma_i\sigma_j\qquad\text{with}\qquad\sigma_i:=\sqrt{2p(1-p)}{e_{\G_i}\over|\aut(\G_i)|}p^{e_{\G_i}-1}\,,\qquad i\in\{1,\ldots,d\}
$$
and let $\bN_\Sig$ denote a $d$-dimensional centred Gaussian vector with covariance matrix $\Sig$.
Then, there exists a constant $c:=c(\G_1,\ldots,\G_d,p)>0$ only depending on $\G_1,\ldots,\G_d$ and on $p$ such that
$$
d_4(\bX_\bfG,\bN_\Sig) \leq {c\over n}
$$
for all sufficiently large $n$.
\end{theorem}
\begin{proof}
It readily follows from Lemma \ref{lem:CovarianceERG} and the definition of the constants $\sigma_i$ in the statement of the theorem that, for all $i,j=1,\ldots,d$, $\cov(\widetilde{X}_{\G_i},\widetilde{X}_{\G_j}) = \Sig_{ij}+O(n^{-1})$ and hence $|\cov(\widetilde{X}_{\G_i},\widetilde{X}_{\G_j})-\Sig_{ij}|=O(n^{-1})$.

To evaluate the other terms in the bound provided by the discrete second-order Poincar\'e inequality in Theorem \ref{thm:SecondOrderPoincare}, we first consider for each $i=1,\ldots,d$ and $k,\ell=1,\ldots,{n\choose 2}$ the first-and second-order discrete Malliavin derivatives $D_k\widetilde{X}_{\G_i}$ and $D_kD_\ell\widetilde{X}_{\G_i}$. From the very definition it follows that
$$
D_k\widetilde{X}_{\G_i} = {\sqrt{p(1-p)}\over n^{v_{\G_i}-1}}((X_{\G_i})_k^+-(X_{\G_i})_k^-)
$$
and the difference $(X_{\G_i})_k^+-(X_{\G_i})_k^-$ is just the number of copies of $\G_i$ that contain edge $e_k$. Since there are $O(n^{v_{\G_i}-2})$ possible choices for the $v_{\G_i}-2$ missing vertices to build such a copy, it follows that $(X_{\G_i})_k^+-(X_{\G_i})_k^-=O(n^{v_{\G_i}-2})$ and thus
\begin{align}\label{eq:ERG1stOrder}
D_k\widetilde{X}_{\G_i}=O(n^{-1})\,.
\end{align}
For the same reason we conclude that
\begin{align}\label{eq:ERG2ndOrder}
D_kD_\ell\widetilde{X}_{\G_i} = \begin{cases}
O(n^{-1}) &: |e_k\cap e_\ell| = 2\,,\\
O(n^{-2}) &: |e_k\cap e_\ell| = 1\,,\\
O(n^{-3}) &: |e_k\cap e_\ell| = 0\,,
\end{cases}
\end{align}
where $|e_k\cap e_\ell|$ denotes the number of vertices that $e_k$ and $e_\ell$ have in common.

We can now start to bound, for each $i,j=1,\ldots,d$, the term $B_1(i,j)$. Using the Cauchy-Schwarz inequality, it first follows that
$$
B_1(i,j)^2 = \frac{15}{4} \sum_{k,\ell,m=1}^{n \choose 2} (\EE[(D_kX_{\Gamma_i})^4]\EE[(D_\ell X_{\Gamma_i})^4])^{1/4} (\EE[(D_mD_kX_{\Gamma_j})^4]\EE[(D_mD_\ell X_{\Gamma_j})^4])^{1/4} \,.
$$

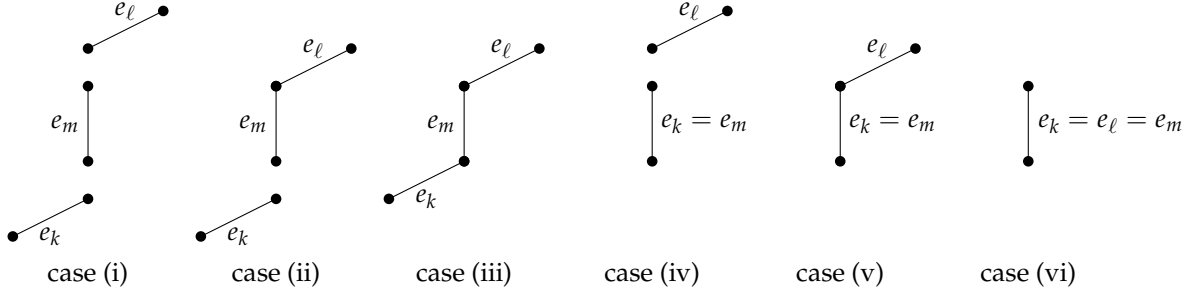
\begin{figure}[t]
\begin{center}
\begin{tikzpicture}
\coordinate (A) at (0,0);
\coordinate (B) at (1,0.5);
\coordinate (C) at (1,1);
\coordinate (D) at (1,2);
\coordinate (E) at (1,2.5);
\coordinate (F) at (2,3);

\fill (A) circle (2pt);
\fill (B) circle (2pt);
\fill (C) circle (2pt);
\fill (D) circle (2pt);
\fill (E) circle (2pt);
\fill (F) circle (2pt);
\draw (A) -- (B);
\draw (C) -- (D);
\draw (E) -- (F);
\node at (0.5,0) {$e_k$};
\node at (0.7,1.5) {$e_m$};
\node at (1.5,3) {$e_\ell$};
\node at (1,-0.5) {case (i)};

\coordinate (A) at (0+2.5,0);
\coordinate (B) at (1+2.5,0.5);
\coordinate (C) at (1+2.5,1);
\coordinate (D) at (1+2.5,2);
\coordinate (E) at (1+2.5,2);
\coordinate (F) at (2+2.5,2.5);

\fill (A) circle (2pt);
\fill (B) circle (2pt);
\fill (C) circle (2pt);
\fill (D) circle (2pt);
\fill (E) circle (2pt);
\fill (F) circle (2pt);
\draw (A) -- (B);
\draw (C) -- (D);
\draw (E) -- (F);
\node at (0.5+2.5,0) {$e_k$};
\node at (0.7+2.5,1.5) {$e_m$};
\node at (1.5+2.5,2.5) {$e_\ell$};
\node at (1+2.5,-0.5) {case (ii)};

\coordinate (A) at (0+5,0.5);
\coordinate (B) at (1+5,1);
\coordinate (C) at (1+5,1);
\coordinate (D) at (1+5,2);
\coordinate (E) at (1+5,2);
\coordinate (F) at (2+5,2.5);

\fill (A) circle (2pt);
\fill (B) circle (2pt);
\fill (C) circle (2pt);
\fill (D) circle (2pt);
\fill (E) circle (2pt);
\fill (F) circle (2pt);
\draw (A) -- (B);
\draw (C) -- (D);
\draw (E) -- (F);
\node at (0.5+5,0.5) {$e_k$};
\node at (0.7+5,1.5) {$e_m$};
\node at (1.5+5,2.5) {$e_\ell$};
\node at (1+5,-0.5) {case (iii)};

\coordinate (C) at (1+7.5,1);
\coordinate (D) at (1+7.5,2);
\coordinate (E) at (1+7.5,2.5);
\coordinate (F) at (2+7.5,3);

\fill (C) circle (2pt);
\fill (D) circle (2pt);
\fill (E) circle (2pt);
\fill (F) circle (2pt);
\draw (C) -- (D);
\draw (E) -- (F);
\node at (1.7+7.5,1.5) {$e_k=e_m$};
\node at (1.5+7.5,3) {$e_\ell$};
\node at (1+7.5,-0.5) {case (iv)};

\coordinate (C) at (1+10,1);
\coordinate (D) at (1+10,2);
\coordinate (E) at (1+10,2);
\coordinate (F) at (2+10,2.5);

\fill (C) circle (2pt);
\fill (D) circle (2pt);
\fill (E) circle (2pt);
\fill (F) circle (2pt);
\draw (C) -- (D);
\draw (E) -- (F);
\node at (1.7+10,1.5) {$e_k=e_m$};
\node at (1.5+10,2.5) {$e_\ell$};
\node at (1+10,-0.5) {case (v)};

\coordinate (C) at (1+12.5,1);
\coordinate (D) at (1+12.5,2);

\fill (C) circle (2pt);
\fill (D) circle (2pt);
\draw (C) -- (D);
\node at (2.1+12.5,1.5) {$e_k=e_\ell=e_m$};
\node at (1+12.5,-0.5) {case (vi)};

\end{tikzpicture}
\end{center}
\caption{The different cases arising in the proof of Theorem \ref{thm:CLTRandomGraphs}. We distinguish according to the number of vertices that edges $e_m$ has in common with $e_k$ and $e_\ell$, respectively. For example, the illustration in case (ii) means that $|e_k\cap e_m|=0$ and $|e_\ell\cap e_m|=1$ or, vice versa, $|e_k\cap e_m|=1$ and $|e_\ell\cap e_m|=0$. This allows both situations, $|e_k\cap e_\ell|=0$ or $|e_k\cap e_\ell|=1$ with $|e_k\cap e_\ell\cap e_m|=0$.}
\label{fig:Illu}
\end{figure}

Now, we have to distinguish different cases that are illustrated in Figure \ref{fig:Illu} (up to permutation of the indices $k,\ell$ and $m$). In case (i), we have $O({n\choose 2}) = O(n^2)$ possibilities to choose each of the three edges and by \eqref{eq:ERG1stOrder} each first-order discrete Malliavin derivative contributes $O(n^{-1})$, while each second-order derivatives contribute $O(n^{-3})$ according to \eqref{eq:ERG2ndOrder}. Thus, in case (i) the sum is of order $O(n^{6}\cdot n^{-2}\cdot n^{-6})=O(n^{-2})$. In case (ii), we have $O({n\choose 2}) = O(n^2)$ possibilities to choose each of the edges $e_k$ and $e_m$, while there are only $O(n)$ possibilities for $e_\ell$. Moreover, in view of \eqref{eq:ERG1stOrder} and \eqref{eq:ERG2ndOrder} the first-order discrete Malliavin derivatives contribute again $O(n^{-2})$, but the second-order derivatives contribute only $O(n^{-5})$. Thus, the terms corresponding to case (ii) in the above sum are of order $O(n^{5}\cdot n^{-2}\cdot n^{-5})=O(n^{-2})$. The same behaviour is also valid for cases (iii), (iv), (v) and (vi), which shows that $B_1(i,j)^2=O(n^{-2})$. Similarly we see that $B_2(i,j)^2=O(n^{-2})$.

We are thus left with the terms $B_3(i,j)$ and $B_4(i,j)$ given by
\begin{align*}
B_3(i,j) &:= \frac{1}{2} d^{3/2} \sum_{k=1}^{n \choose 2} \frac{|p_k-q_k|}{\sqrt{p_kq_k}} (\EE[(D_kX_{\Gamma_i})^2])^{1/2} (\EE[(D_kX_{\Gamma_j})^4])^{1/2}\,,\\
B_4(i,j) &:= \frac{5}{12} d^2 \sum_{k=1}^{n \choose 2} \frac{1}{p_kq_k} (\EE[(D_kX_{\Gamma_i})^4])^{1/4} (\EE[(D_kX_{\Gamma_j})^4])^{3/4}\,.
\end{align*}
In $B_3(i,j)$ there are ${n\choose 2}\asymp n^2$ choices for $k$ and the first-order discrete Malliavin derivatives are of order $O(n^{-1})$ by \eqref{eq:ERG1stOrder}, which shows that $B_3(i,j)=O(n^{-1})$ for all choices of $i$ and $j$. Finally, in $B_4(i,j)$ there are again ${n\choose 2}\asymp n^2$ choices for $k$ and once again by \eqref{eq:ERG1stOrder} the derivatives are of order $O(n^{-1})$. Hence, $B_4(i,j)=O(n^{-2})$ for all possible choices of $i$ and $j$. Summarizing we conclude that 
$$
d_4(\bX_{\bf\Gamma},\bN_\Sig)\leq\frac{1}{2} \sum_{i,j=1}^d \big[ |\Sigma_{ij}-\cov(X_{\Gamma_i},X_{\Gamma_j})| + B_1(i,j) + B_2(i,j) + B_3(i,j) + B_4(i,j) \big] = O(n^{-1})\,,
$$
where the constant hidden in the $O$-notation only depends on $p$ and the graphs $\G_1,\ldots,\G_d$. This completes the proof of the theorem.
\end{proof}

\begin{remark}\label{rem:PositiveDefiniteERG}\rm
The structure of the asymptotic covariance matrix $\Sig$ in the previous theorem implies that $\Sig$ has rank $1$. Thus, $\Sig$ cannot be positive definite, but it clearly is positive semidefinite.
\end{remark}

\begin{remark}\rm 
We believe that there are also other methods available in the existing literature that allow to prove results similar to Theorem \ref{thm:CLTRandomGraphs}. For example, the multivariate exchangeable pairs approach used in \cite{ReinertRollin} might be generalized to subgraph counts of arbitrary graphs. On the other hand, this might require serious technical efforts, while our proof of the quantitative multivariate central limit theorem for subgraph counts basically only requires simple (asymptotic) counting arguments. A similar comment also applies to the random cubical complexes treated in the next section.
\end{remark}

We continue our study of the Erd\H{o}s-R\'enyi random graph $\cG(n,p)$ by establishing a central limit theorem for the vertex degree statistic in the the case that $p=\theta/(n-1)$ for a $\theta \in (0,1)$. Although the number of vertices of a given degree is a special case of a subgraph counting statistic as considered above, the significant difference here is that we allow the success probability $p$ to vary with $n$.

For $i\geq 0$ we denote by $V_i$ the number of vertices of degree $i$ in the Erd\H{o}s-R\'enyi random graph $\cG(n,p)$ for a $p\in(0,1)$ and where we assume that $n$ is sufficiently large so that all quantities we deal with are well-defined. More formally, if we denote by $v_1,\ldots,v_n$ the $n$ vertices of the complete graph $K_n$, then
$$
V_i = \sum_{k=1}^n{\bf 1}\{\deg(v_k)=i\}\,,
$$
where $\deg(v_k)$ is the degree of $v_k$ in $\cG(n,p)$, that is, the (random) number of edges emanating from $v_k$. Since for each $k\in\{1,\ldots,n\}$, $\EE[{\bf 1}\{\deg(v_k)=i\}]=P(\deg(v_k)=i)={n-1\choose i}p^i(1-p)^{n-1-i}$, it follows that
$$
\EE[V_i] = n{n-1\choose i}p^i(1-p)^{n-1-i}\,.
$$
The covariance between $V_i$ and $V_j$ for $i,j\geq 0$ under the choice $p=\theta/(n-1)$ for a $\theta\in(0,1)$ has been investigated in \cite{GoldsteinRinott} and we recall from Theorem 4.2 there that
\begin{equation}\label{eq:CovDegrees}
\cov(V_i,V_j) = {1\over n}\EE[V_i]\EE[V_j]\Big({(i-\theta)(j-\theta)\over \theta(1-\theta/(n-1))}-1\Big)+{\bf 1}\{i=j\}\EE[V_i]\,.
\end{equation}

We define $F_i:=(V_i-\EE[V_i])/\sqrt{n}$, fix $d\geq 1$ as well as $1\leq i_1<\ldots<i_d$, put $\bD:=(i_1,\ldots,i_d)$ and define the random vector $\bF_\bD:=(F_{i_1},\ldots,F_{i_d})$. From now on and for the rest of this subsection, we assume that the success probability $p$ is of the form $p=\theta/(n-1)$ for some fixed $\theta\in(0,1)$. Then, it is easily seen from the expression for $\cov(V_i,V_j)$ in \eqref{eq:CovDegrees} that
\begin{equation}\label{eq:CovVertexDegrees}
\cov(F_i,F_j) = {1\over n}\cov(V_i,V_j) \to {\theta^{i+j}\over i!\,j!}e^{-2\theta}\Big({(i-\theta)(j-\theta)\over\theta}-1\Big)+{\bf 1}\{i=j\}{\theta^i\over i!}e^{-\theta}\,,
\end{equation}
as $n\to\infty$. Our next result is a multivariate central limit theorem for the vertex degree vector $\bF_\bD$. It is a version of \cite[Theorem 4.2]{GoldsteinRinott} for which, using a slightly smoother probability metric, we can give a quick proof based on our multivariate discrete second-order Poincar\'e inequality. For Berry-Essen-type rates of convergence in the one-dimensional case we refer to \cite[Theorem 2.1]{Goldstein}, \cite[Theorem 1.3]{KRT2} and \cite[Theorem 2.1 and Equation (3.3)]{Raic}.

\begin{theorem}\label{thm:CLTDegrees}
Let $\Sig=(\Sig_{ij})_{i,j=1}^d$ be the matrix given by
$$
\Sig_{ij}={\theta^{i+j}\over i!\,j!}e^{-2\theta}\Big({(i-\theta)(j-\theta)\over\theta}-1\Big)+{\bf 1}\{i=j\}{\theta^i\over i!}e^{-\theta}\,,\qquad i\in\{1,\ldots,d\}\,,
$$
and let $\bN_\Sig$ be a $d$-dimensional centred Gaussian random vector with covariance matrix $\Sig$. Then, there exists a constant $c=c(i_1,\ldots,i_d,\theta)>0$ only depending on $i_1,\ldots,i_d$ and $\theta$ such that
$$
d_4(\bF_\bD,\bN_\Sig)\leq{c\over\sqrt{n}}
$$
for all sufficiently large $n$.
\end{theorem}
\begin{proof}
From \eqref{eq:CovVertexDegrees} we infer that $\cov(F_i,F_j)\to\Sig_{ij}$, as $n\to\infty$. Moreover, from the structure of the covariance \eqref{eq:CovDegrees} we also conclude that $|\cov(F_i,F_j)-\Sig_{ij}|=O(n^{-1})$. 

Next, we fix $i\in\{1,\ldots,d\}$ and $k,\ell\in\{1,\ldots,{n\choose 2}\}$. As in the proof of Theorem 1.3 in \cite{KRT2} we notice that adding or removing an edge from $\cG(n,p)$ results in a change of at most $2$ for the number of vertices of degree $i$. In other words,
$$
|D_kF_i| \leq {2\sqrt{pq}\over\sqrt{n}}\,.
$$
For the second-order discrete Malliavin derivative we observe that $D_kD_\ell F_i$ is zero whenever $k=\ell$ or the edges $e_k$ and $e_\ell$ corresponding to $k$ and $\ell$, respectively, do not have a common vertex. Thus, it follows that
$$
|D_kD_\ell F_i|\leq {2pq\over\sqrt{n}}\,{\bf 1}\{|e_k\cap e_\ell|=1\}\,.
$$
We can now evaluate the terms $B_1(i,j)$ to $B_4(i,j)$ in Theorem \ref{thm:SecondOrderPoincare}. Using the Cauchy-Schwarz inequality we first conclude that
\begin{align*}
B_1(i,j)^2 &= \frac{15}{4} \sum_{k,\ell,m=1}^{n\choose 2} (\EE[(D_kF_i)^4]\EE[(D_\ell F_i)^4])^{1/4} (\EE[(D_mD_kF_j)^4]\EE[(D_mD_\ell F_j)^4])^{1/4} \\
&\leq {60(pq)^3\over n^2}\sum_{k,\ell,m=1}^{n\choose 2}{\bf 1}\{|e_m\cap e_k|=1,|e_m\cap e_\ell|=1\}\\
& = {60(pq)^3\over n^2}{n\choose 2}(2(n-2))^2 = O(n^{-1})\,,
\end{align*}
since $p=\theta/(n-1)$. Similarly, we have that $B_2(i,j)^2=O(n^{-1})$. For the remaining terms $B_3(i,j)$ and $B_4(i,j)$ we see that
\begin{align*}
B_3(i,j) &= \frac{1}{2} d^{3/2} \sum_{k=1}^{n\choose 2} \frac{|p_k-q_k|}{\sqrt{p_kq_k}} (\EE[(D_kF_i)^2])^{1/2} (\EE[(D_kF_j)^4])^{1/2} \\
&\leq  {n \choose 2} \frac{4pq}{n^{3/2}} d^{3/2} = O(n^{-1/2})\,,\\
B_4(i,j) &= \frac{5}{12} d^2 \sum_{k=1}^{n\choose 2} \frac{1}{p_kq_k} (\EE[(D_kF_i)^4])^{1/4} (\EE[(D_kF_j)^4])^{3/4}\\
&\leq {n \choose 2} \frac{20pq}{3n^2} d^2 = O(n^{-1})\,.
\end{align*}
By Theorem \ref{thm:SecondOrderPoincare} we have thus proved the result.
\end{proof}

\subsection{Intrinsic volumes of random cubical complexes}

\begin{figure}[t]
\begin{center}
\begin{tikzpicture}
\draw [dashed] (0,0) rectangle (4,4);
\draw [fill=lightgray] (2,0) rectangle (3,1);

\draw [fill=lightgray] (1,1) rectangle (2,2);

\draw [fill=lightgray] (1,2) rectangle (2,3);

\draw [fill=lightgray] (0,3) rectangle (1,4);
\draw [fill=lightgray] (3,3) rectangle (4,4);


\draw [dashed] (0+5,0) rectangle (4+5,4);
\draw [fill=lightgray] (0+5,0) rectangle (1+5,1);
\draw [fill=lightgray] (3+5,0) rectangle (4+5,1);

\draw [fill=lightgray] (2+5,1) rectangle (3+5,2);

\draw [fill=lightgray] (1+5,2) rectangle (2+5,3);
\draw [fill=lightgray] (2+5,2) rectangle (3+5,3);

\draw [fill=lightgray] (0+5,3) rectangle (1+5,4);
\draw [fill=lightgray] (2+5,3) rectangle (3+5,4);
\draw [fill=lightgray] (3+5,3) rectangle (4+5,4);


\draw [dashed] (0+10,0) rectangle (4+10,4);
\draw [fill=lightgray] (1+10,0) rectangle (2+10,1);
\draw [fill=lightgray] (2+10,0) rectangle (3+10,1);
\draw [fill=lightgray] (3+10,0) rectangle (4+10,1);

\draw [fill=lightgray] (0+10,1) rectangle (1+10,2);
\draw [fill=lightgray] (1+10,1) rectangle (2+10,2);
\draw [fill=lightgray] (3+10,1) rectangle (4+10,2);

\draw [fill=lightgray] (1+10,2) rectangle (2+10,3);
\draw [fill=lightgray] (2+10,2) rectangle (3+10,3);
\draw [fill=lightgray] (3+10,2) rectangle (4+10,3);

\draw [fill=lightgray] (0+10,3) rectangle (1+10,4);
\draw [fill=lightgray] (1+10,3) rectangle (2+10,4);
\end{tikzpicture}
\end{center}
\caption{Illustrations of the voxel model $\cC$ of a random cubical complex with $d=2$ and $n=4$ for increasing values of $p$.}
\label{fig:complexes}
\end{figure}
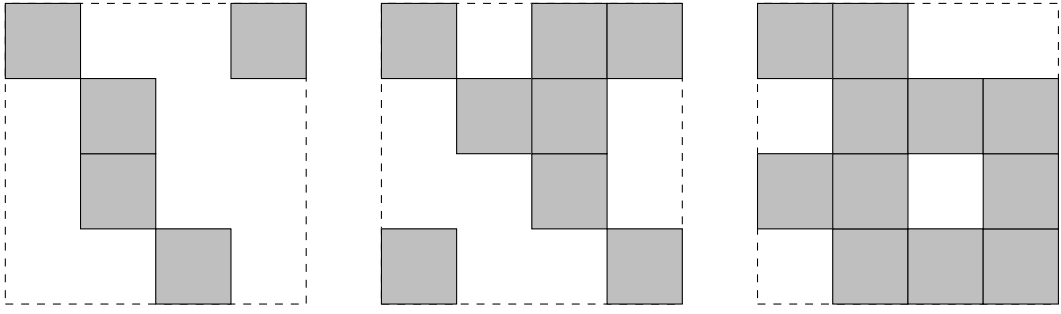

Fix a space dimension $d\geq 1$, $n \geq 3$ and consider the lattice $\cL:=\{[0,1]^d+z:z\in\{0,\ldots,n-1\}^d\}$ consisting of $n^d$ unit cubes $C_1, \dotsc, C_{n^d}$ of dimension $d$. To avoid boundary effects, we identify opposite faces in $\cL$, a convention which supplies $\cL$ with the topology of a $d$-dimensional torus. Now, we number the cubes in $\cL$ in a fixed but arbitrary way and assign to each cube $C_k\in\cL$ a Rademacher random variable $X_k$ such that $P(X_k=1)=p$ and $P(X_k=-1)=1-p=:q$ for some fixed parameter $p\in(0,1)$. Following the paper of Werman and Wright \cite{WermanWright} the voxel model for a so-called random cubical complex $\cC$ arises from $\cL$ when each cube $C_k$ is removed from $\cL$ for which $X_k=-1$, see Figure \ref{fig:complexes}. It should be clear that the random cubical complexes arising in this way may be represented as finite unions of disjoint open cubes of dimensions $0,1,\ldots,d$, corresponding to the vertices, edges, etc.

We are interested in the intrinsic volumes $V_j(\cC)$ of the random cubical complex $\cC$ for all $j\in\{0,1,\ldots,d\}$. To define these quantities formally one can follow approach from \cite{Groemer}, where Groemer introduced a way to define intrinsic volumes for the relative interior of a convex body. Since we are interested only in finite unions of cubes, we go the more direct way also used in \cite{WermanWright}. Namely, for a $\d \in \NN$, by a closed $\d$-cube we understand any translate of $[0,1]^\d$, while an open $\d$-cube refers to a translate of $(0,1)^\d$. The intrinsic volume $V_j(C)$ of order $j\in\{0,1,\ldots,i\}$ of a closed $\d$-cube $C$ is given by $V_j(C)={\d\choose j}$, while the $j$th intrinsic volume of an open $\d$-cube $D$ is $V_j(D)=(-1)^{\d-j}{\d\choose j}=:V_j(\d)$. Finally, for the random cubical complex $\cC$ as defined above we have the following representation for $V_j(\cC)$ from \cite{WermanWright}:
\begin{align}\label{eq:DefVw}
V_j(\cC) = \sum_{D\text{ open cube in }\cL}\xi_{D,j}\,,\qquad\xi_{D,j}:=(-1)^{\dim(D)-j}{\dim(D)\choose j}\,{\bf 1}\{D\text{ belongs to }\cC\}\,.
\end{align}
From this representation it readily follows that
$$
\EE[V_j(\cC)] = \sum_{D\text{ open cube in }\cL}\EE[\xi_{D,j}] = \sum_{\d=j}^dN_\d\,P_\d\,V_j(\d)\,,
$$
where $N_\d={d\choose \d}n^d$ denotes the number of $\d$-cubes in $\cL$ and $P_\d=1-q^{2^{d-\d}}$ is the probability that an arbitrary $\d$-cube is included in $\cC$, see \cite{WermanWright}. 

Although the variance of $V_j(\cC)$ has been computed in \cite{WermanWright}, in our context we will also need information about the covariance structure between $V_i(\cC)$ and $V_j(\cC)$. This is provided by the next lemma.

\begin{lemma}\label{lem:covComplexes}
Let $i,j\in\{0,1,\ldots,d\}$. Then,
\begin{align*}
\cov(V_i(\cC),V_j(\cC)) = c(i,j)\,n^d
\end{align*}
with the constant $c(i,j)$ given by
$$
c(i,j) = \sum_{a=0}^d\sum_{b=0}^d\sum_{\d=0}^d V_i(a)V_j(b) {d \choose \d} N_{a,b,\d}\,q^{2^{d-a}+2^{d-b}}(q^{-{2^{d-\d}}}-1)\,,
$$
where $N_{a,b,\d} = \sum\limits_{\ell=0}^\d(-1)^{\d-\ell}{\d\choose\ell}{\ell\choose a}{\ell\choose b}2^{\d+\ell-a-b}$.
\end{lemma}
\begin{proof}
We first notice that for two open cubes $D$ and $D'$ in $\cL$ (possibly having different dimensions) the random variables $\xi_D$ and $\xi_{D'}$ are independent whenever $D$ and $D'$ are not faces of a common $d$-dimensional cube from $\cL$. Thus, using \eqref{eq:DefVw} we conclude that
\begin{align*}
\cov(V_i(\cC),V_j(\cC)) = \sum_{D,D'}\cov(\xi_{D,i},\xi_{D',j}) = \sum_{D,D'}\big(\EE[\xi_{D,i}\xi_{D',j}]-\EE[\xi_{D,i}]\EE[\xi_{D',j}]\big)
\end{align*}
with the sum running over all open cubes $D,D'$ in $\cL$ that are faces of a common $d$-cube. To evaluate this sum, we observe that for each pair of cubes $D,D'$ there is a unique cube $C(D,D')$ of which $D$ and $D'$ are common faces and which has the smallest dimension among all such cubes (in fact, the existence of such a cube is the reason why $n\geq 3$ is assumed in this section). On the contrary, if $C$ is a cube of dimension $\d\in\{0,1,\ldots,d\}$, we let $N_{a,b,\d}$ be the number of pairs of cubes $D$ and $D'$ of dimensions $a$ and $b$, respectively, for which $C(D,D')=C$. We notice that the value of $N_{a,b,\d}$ is independent of the particular choice of $C$ and given by
$$
N_{a,b,\d} = \sum_{\ell=0}^\d(-1)^{\d-\ell}{\d\choose\ell}{\ell\choose a}{\ell\choose b}2^{\d+\ell-a-b}
$$
according to Equation (18) in \cite{WermanWright}. Especially, $N_{a,b,\d}$ is independent of $n$. Moreover, following Equation (20) in \cite{WermanWright} we denote by
$$
P_{a,b,\d} = (1-q^{2^{d-\d}})+q^{2^{d-\d}}(1-q^{2^{d-a}-2^{d-\d}})(1-q^{2^{d-b}-2^{d-\d}})
$$
the probability that both $D$ and $D'$ are included in the cubical complex $\cC$. Then, we conclude that
\begin{align*}
\cov(V_i(\cC),V_j(\cC)) &= \sum_{a=0}^d\sum_{b=0}^d\sum_{\d=0}^d N_\d N_{a,b,\d}\big(\EE[\xi_{D,i}\xi_{D',j}]-\EE[\xi_{D,i}]\EE[\xi_{D',j}]\big)\,.
\end{align*}
According to our above discussion, the two expectations $\EE[\xi_{D,i}]$ and $\EE[\xi_{D',j}]$ are given by $\EE[\xi_{D,i}]=P_aV_i(a)$ and $\EE[\xi_{D',j}]=P_bV_j(b)$. Finally, $\EE[\xi_{D,i}\xi_{D',j}]$ equals $P_{a,b,\d}V_i(a)V_j(b)$, which implies that
\begin{align*}
\cov(V_i(\cC),V_j(\cC)) &= \sum_{a=0}^d\sum_{b=0}^d\sum_{\d=0}^d V_i(a)V_j(b)N_\d N_{a,b,\d}(P_{a,b,\d}-P_aP_b)\,.
\end{align*}
Since $P_{a,b,\d}-P_aP_b=q^{2^{d-a}+2^{d-b}}(q^{-{2^{d-\d}}}-1)$, the proof is complete.
\end{proof}

Now, define for $j\in\{0,1,\ldots,d\}$ the centred and normalized random variables $\widetilde{V}_j(\cC):=n^{-d/2}(V_j(\cC)-\EE[V_j(\cC)])$ as well as the random vector $\bV:=(\widetilde{V}_0(\cC),\widetilde{V}_1(\cC),\ldots,\widetilde{V}_d(\cC))$. Our next theorem provides a bound for the multivariate normal approximation of $\bV$ and this way extends Theorem 4 in \cite{WermanWright}.

\begin{theorem}\label{thm:Complexes}
Let $\Sig:=(\Sig_{ij})_{i,j=0}^d$ be the matrix $\Sig_{ij}:=c(i,j)$ with the constants $c(i,j)$ given by Lemma \ref{lem:covComplexes}. Then, there exists a constant $C=C(p,d)$ only depending on $p$ and on $d$ such that
$$
d_4(\bV,\bN_\Sig)\leq {C\over n^{d/2}}\,,
$$
where $\bN_\Sig$ is a $(d+1)$-dimensional centred Gaussian vector with covariance matrix $\Sig$.
\end{theorem}
\begin{proof}
By Lemma \ref{lem:covComplexes} it follows that, for all $i,j\in\{0,1,\ldots,d\}$, $\cov(\widetilde{V}_i(\cC),\widetilde{V}_j(\cC))=\Sig_{ij}$. Thus, it only remains to bound the terms $B_1(i,j)$ to $B_4(i,j)$ in Theorem \ref{thm:SecondOrderPoincare}. To this end, we need appropriate estimates for the first- and second-order discrete Malliavin derivatives $D_k\widetilde{V}_i(\cC)$ and $D_kD_\ell\widetilde{V}_i(\cC)$ for all $i\in\{0,1,\ldots,d\}$, respectively. For this, we recall the representation \eqref{eq:DefVw} and observe that for each $k\in\{1,\ldots,n^d\}$, $D_k\widetilde{V}_i(\cC)$ can be written as $\sqrt{pq}/n^{d/2}$ times a sum of at most $6^d$ summands, where each of them is bounded independently of $n$. Here, $6^d\geq 2^{d-\d}\cdot 3^d$ for any $\d\in\{0,1,\ldots,d\}$ and $2^{d-\d}$ is the number of $d$-dimensional cubes of which a fixed $\d$-dimensional cube is a face of, while $3^d=\sum_{\d=0}^d{d\choose \d}2^{d-\d}$ is the total number of faces of a $d$-dimensional cube. As a consequence, we find that
$$
D_k\widetilde{V}_i(\cC) = O(n^{-d/2})
$$
and by the triangle inequality also
$$
D_kD_\ell\widetilde{V}_i(\cC) = O(n^{-d/2})\,,
$$
where the hidden constants only depend on $d$ and on $p$. Now, it is crucial to observe that for any fixed $k\in\{1,\ldots,n^d\}$ the second-order discrete Malliavin derivative $D_kD_\ell\widetilde{V}_i(\cC)$ is even identically zero whenever the cubes corresponding to $k$ and $\ell$ are not neighbours of each other. Since any cube in $\cL$ has only a finite number of neighbours, independently of $n$, we conclude that in the term $B_1(i,j)$ provided by the multivariate discrete second-order Poincar\'e inequality in Theorem \ref{thm:SecondOrderPoincare} there are exactly $n^d$ choices for $m$ and only a constant number of choices for $k$ and $\ell$ for which the corresponding summand is non-vanishing. As a consequence, $B_1(i,j)^2$ is of order $O(n^d\cdot n^{-d}\cdot n^{-d})=O(n^{-d})$, implying that $B_1(i,j)=O(n^{-d/2})$ for any choice of $i,j\in\{0,1,\ldots,d\}$. Since the same behaviour can also be observed for the remaining terms $B_2(i,j),B_3(i,j)$ and $B_4(i,j)$, the claim follows.
\end{proof}

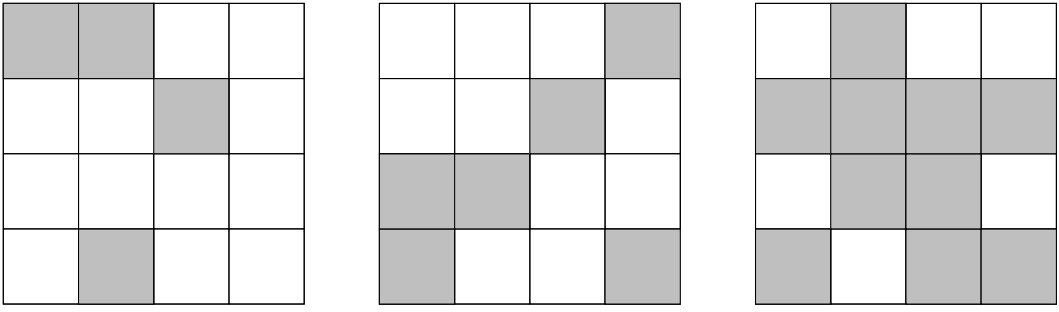
\begin{figure}[t]
\begin{center}
\begin{tikzpicture}
\draw (0,0) rectangle (4,4);
\draw (0,0) rectangle (1,1);[fill=lightgray]
\draw (1,0) rectangle (2,1);
\draw [fill=lightgray] (1,0) rectangle (2,1);
\draw (2,0) rectangle (3,1);
\draw (3,0) rectangle (4,1);

\draw (0,1) rectangle (1,2);
\draw (1,1) rectangle (2,2);
\draw (2,1) rectangle (3,2);
\draw (3,1) rectangle (4,2);

\draw  (0,2) rectangle (1,3);
\draw  (1,2) rectangle (2,3);
\draw  (2,2) rectangle (3,3);
\draw [fill=lightgray] (2,2) rectangle (3,3);
\draw  (3,2) rectangle (4,3);

\draw  (0,3) rectangle (1,4);
\draw [fill=lightgray] (0,3) rectangle (1,4);
\draw  (1,3) rectangle (2,4);
\draw [fill=lightgray] (1,3) rectangle (2,4);
\draw  (2,3) rectangle (3,4);
\draw  (3,3) rectangle (4,4);


\draw (0+5,0) rectangle (4+5,4);
\draw  (0+5,0) rectangle (1+5,1);
\draw [fill=lightgray] (5,0) rectangle (6,1);
\draw  (1+5,0) rectangle (2+5,1);
\draw  (2+5,0) rectangle (3+5,1);
\draw  (3+5,0) rectangle (4+5,1);
\draw [fill=lightgray] (8,0) rectangle (9,1);

\draw  (0+5,1) rectangle (1+5,2);
\draw [fill=lightgray] (5,1) rectangle (6,2);
\draw  (1+5,1) rectangle (2+5,2);
\draw [fill=lightgray] (6,1) rectangle (7,2);
\draw  (2+5,1) rectangle (3+5,2);
\draw  (3+5,1) rectangle (4+5,2);

\draw  (0+5,2) rectangle (1+5,3);
\draw  (1+5,2) rectangle (2+5,3);
\draw  (2+5,2) rectangle (3+5,3);
\draw [fill=lightgray] (7,2) rectangle (8,3);
\draw  (3+5,2) rectangle (4+5,3);

\draw  (0+5,3) rectangle (1+5,4);
\draw  (1+5,3) rectangle (2+5,4);
\draw  (2+5,3) rectangle (3+5,4);
\draw  (3+5,3) rectangle (4+5,4);
\draw [fill=lightgray] (8,3) rectangle (9,4);


\draw (0+10,0) rectangle (4+10,4);
\draw  (0+10,0) rectangle (1+10,1);
\draw [fill=lightgray] (10,0) rectangle (11,1);
\draw  (1+10,0) rectangle (2+10,1);
\draw  (2+10,0) rectangle (3+10,1);
\draw [fill=lightgray] (12,0) rectangle (13,1);
\draw  (3+10,0) rectangle (4+10,1);
\draw [fill=lightgray] (13,0) rectangle (14,1);

\draw  (0+10,1) rectangle (1+10,2);
\draw  (1+10,1) rectangle (2+10,2);
\draw [fill=lightgray] (11,1) rectangle (12,2);
\draw  (2+10,1) rectangle (3+10,2);
\draw [fill=lightgray] (12,1) rectangle (13,2);
\draw  (3+10,1) rectangle (4+10,2);

\draw  (0+10,2) rectangle (1+10,3);
\draw [fill=lightgray] (10,2) rectangle (11,3);
\draw  (1+10,2) rectangle (2+10,3);
\draw [fill=lightgray] (11,2) rectangle (12,3);
\draw  (2+10,2) rectangle (3+10,3);
\draw [fill=lightgray] (12,2) rectangle (13,3);
\draw  (3+10,2) rectangle (4+10,3);
\draw [fill=lightgray] (13,2) rectangle (14,3);

\draw  (0+10,3) rectangle (1+10,4);
\draw  (1+10,3) rectangle (2+10,4);
\draw [fill=lightgray] (11,3) rectangle (12,4);
\draw  (2+10,3) rectangle (3+10,4);
\draw  (3+10,3) rectangle (4+10,4);
\end{tikzpicture}
\end{center}
\caption{Illustrations of the plaquette model $\cP$ of a random cubical complex with $d=2$ and $n=4$ for increasing values of $p$. The grey cubes are included, while the white cubes are not included in $\cP$.}
\label{fig:complexes2}
\end{figure}

Besides of the voxel model, the authors of \cite{WermanWright} also consider three further models for random cubical complexes: the plaquette model, the closed faces model and the independent faces model. For each of these models our method can be used to derive a multivariate central limit theorem for the random vector of their intrinsic volumes and to obtain bounds on the $d_4$-distance of order $O(n^{-d/2})$ in each case. We present the result only in the case of the plaquette model, since it is close in spirit to the celebrated random simplicial complexes introduced by Linial and Meshulam \cite{LinialMeshulam} that have been object of intensive studies. To introduce the model formally, we fix $d\geq 1$, $n\geq 3$, and define the set $\cG:=\{\partial[0,1]^d+z:z\in\{0,\ldots,n-1\}^d\}$, where $\partial[0,1]^d$ stands for the boundary of the unit $d$-cube $[0,1]^d$. The open cubes $C_1, \dotsc, C_{n^d}$ in $\{(0,1)^d+z:z\in\{0,\ldots,n-1\}^d\}$ are assumed to be numbered in a fixed but arbitrary way and we assign to each cube $C_k$ a Rademacher random variable $X_k$ with $P(X_k=1)=p$ and $P(X_k=-1)=1-p=:q$. The plaquette model now arises if those open cubes $C_k$ are joint with the set $\cG$ for which the associated Rademacher random variable $X_k$ takes the value $1$, see Figure \ref{fig:complexes2}. 

The construction just described gives rise to a random set $\cP$ and as in the case of the voxel model $\cC$ we are interested in its intrinsic volumes $V_j(\cP)$, $j\in\{0,1,\ldots,d\}$. Using the same notation as in the previous example, we formally have that $V_j(\cP)=\sum_D\xi_D$ with the sum running over all open cubes in $\cP$ and hence $\EE[V_j(\cP)]=\sum_{\d=j}^dN_\d P_\d V_j(\d)$. However, in the plaquette model we have that the probabilities $P_0,P_1\ldots,P_d$ satisfy $P_d=p$ and $P_\d=1$ for $\d \in \{ 0, \dotsc, d-1 \}$, which implies (after some simplifications) that
$$
\EE[V_j(\cP)] = \begin{cases}
pn^d &: j=d\,,\\
(-1)^{d-j}{d\choose j}(p-1)n^d &: j \in \{ 0, \dotsc, d-1 \}\,,
\end{cases}
$$
see also Equation (27) in \cite{WermanWright}. The covariance structure of the intrinsic volumes for the plaquette model is described in the next lemma.

\begin{lemma}\label{lem:CovariancePlaquette}
Let $i,j\in\{0,1,\ldots,d\}$. Then,
$$
\cov(V_i(\cP),V_j(\cP)) = {d\choose i}{d\choose j}p(1-p)\,n^d\,.
$$
\end{lemma}
\begin{proof}
By definition it follows that
$$
\cov(V_i(\cP),V_j(\cP)) = \sum_{D,D'}\cov(\xi_{D,i},\xi_{D',j})\,,
$$
again with the sum running over all open cubes $D$ and $D'$ in $\cP$. We notice that in this model the random variables $\xi_{D,i}$ and $\xi_{D',j}$ are independent except if $D=D'$ and $\dim(D)=\dim(D')=d$. In this case, we clearly have
$$
\cov(\xi_{D,i},\xi_{D',j})=\EE[\xi_{D,i}\xi_{D,j}]-\EE[\xi_{D,i}]\EE[\xi_{D,j}]={d\choose i}{d\choose j}p-{d\choose i}p\cdot{d\choose j}p={d\choose i}{d\choose j}p(1-p)
$$
and the result follows.
\end{proof}

Now, we define the centred and normalized random variables $\widehat{V}_j(\cP):=n^{-d/2}(V_j(\cP)-\EE[V_j(\cP)])$ and the $(d+1)$-dimensional random vector $\bW:=(\widehat{V}_0(\cP),\widehat{V}_1(\cP),\ldots,\widehat{V}_d(\cP))$. The next result is a multivariate central limit theorem for the random vector $\bW$. Since the arguments
are the same as in the proof of Theorem \ref{thm:Complexes}, we have decided not to present the details.

\begin{theorem}\label{thm:Complexes2}
Let $\Sig:=(\Sig_{ij})_{i,j=0}^d$ be the matrix given by $\Sig_{ij}={d\choose i}{d\choose j}p(1-p)$. Then, there exists a constant $C=C(p,d)$ only depending on $p$ and on $d$ such that
$$
d_4(\bW,\bN_\Sig) \leq {C\over n^{d/2}}
$$
with a $(d+1)$-dimensional centred Gaussian random vector $\bN_\Sig$ having covariance matrix $\Sig$.
\end{theorem}

\begin{remark}\rm 
As for subgraph counting statistics it follows from the structure of the asymptotic covariance matrices $\Sig$ in Theorems \ref{thm:Complexes} and \ref{thm:Complexes2} that $\Sig$ only has rank $1$ and is hence only positive semidefinite rather than positive definite.
\end{remark}

\end{document}